\documentclass{amsart}
\usepackage{graphics}
\usepackage{graphicx}
\usepackage{epsfig}
\usepackage{amscd}
\usepackage{amssymb}
\usepackage{amsthm}
\usepackage{amsmath}
\numberwithin{equation}{section}
\usepackage{mathrsfs}
\usepackage{color}
\usepackage{hyperref}
\usepackage{bm}
\usepackage{float}
\usepackage{geometry}
\usepackage{tikz}
\usepackage{multirow}
\usepackage{verbatim}
\usepackage{geometry}
\usepackage{placeins}

\usetikzlibrary{matrix,arrows,arrows.meta,calc}
\usetikzlibrary{decorations.pathreplacing,decorations.markings}
\usetikzlibrary{shapes,positioning}
\tikzstyle arrowstyle=[scale=1]
\tikzstyle directed=[postaction={decorate,decoration={markings,
    mark=at position .5 with {\arrow[arrowstyle]{stealth}}}}]
\tikzstyle reverse directed=[postaction={decorate,decoration={markings,
    mark=at position .65 with {\arrowreversed[arrowstyle]{stealth};}}}]

\newtheorem{thm}{Theorem}[section]
\newtheorem{cor}[thm]{Corollary}
\newtheorem{prop}[thm]{Proposition}
\newtheorem{conj}[thm]{Conjecture}

\theoremstyle{plain}
\newtheorem{lem}[thm]{Lemma}

\theoremstyle{remark}
\newtheorem{rem}[thm]{Remark}

\theoremstyle{definition}
\newtheorem{defn}[thm]{Definition}
\newtheorem{ex}[thm]{Example}

\numberwithin{equation}{section}

\newcommand{\A}{\mathcal{A}}
\newcommand{\T}{\mathcal{T}}
\newcommand{\B}{\mathcal{B}}

\newcommand{\FF}{\mathcal{F}}

\newcommand{\Q}{\mathbb{Q}}
\newcommand{\R}{\mathbb{R}}
\newcommand{\Z}{\mathbb{Z}}

\newcommand{\id}{\text{id}}
\newcommand{\geod}{\mathrm{geod}}
\newcommand{\gs}{\mathrm{gs}}
\newcommand{\Mag}{\mathrm{Mag}}
\newcommand{\Ch}{\mathrm{Ch}}
\newcommand{\supp}{\mathrm{supp}}
\newcommand{\house}{\mathrm{house}}
\DeclareMathOperator{\rank}{rank}
\DeclareMathOperator{\diam}{diam}
\DeclareMathOperator{\codim}{codim}

\begin{document}

\title{Magnitude homology of real hyperplane arrangements}

\author{Junnosuke Koizumi}
\address{RIKEN iTHEMS, Wako, Saitama 351-0198, Japan}
\email{junnosuke.koizumi@riken.jp}
%\thanks{}

\author{Ye Liu}
\address{Department of Applied Mathematics, Xi'an Jiaotong-Liverpool University, Suzhou, Jiangsu 215123, P. R. China}
\email{yeliumath@gmail.com}
%\thanks{}

%    General info

\date{}

%\dedicatory{}


\begin{abstract}
We define and study the magnitude and magnitude homology of real hyperplane arrangements by regarding their tope graphs as metric spaces.
Since magnitude is a refinement of the cardinality, this may be viewed as a natural generalization of the classical problem of counting the chambers of an arrangement.
Our central result is a structure theorem asserting that the magnitude and magnitude homology of an arrangement decompose into contributions indexed by its faces.
In particular, the magnitude is determined by the intersection lattice of the arrangement.
We also observe several interesting phenomena, including a characterization of Boolean arrangements in terms of diagonality and a reciprocity theorem reminiscent of Ehrhart–Macdonald reciprocity.
Finally, based on explicit computations, we conjecture that the magnitude homology of an arrangement is torsion-free and determined by its intersection lattice.
\end{abstract}

\maketitle

\tableofcontents

\section{Introduction}

Let $\A$ be a finite arrangement of hyperplanes in $\R^d$ passing through the origin.
The hyperplanes of $\A$ divide $\R^d$ into several regions called chambers.
We denote the set of chambers by $\Ch(\A)$.
The collection of all subspaces that can be expressed as intersections of hyperplanes in $\A$ forms a lattice $L(\A)$ under reverse inclusion.
This is called the intersection lattice of $\A$.
Using the M\"obius function of the intersection lattice, one defines a polynomial $\chi_\A(t)$ called the characteristic polynomial of $\A$.
Zaslavsky proved the following formula for the cardinality of $\Ch(\A)$:

\begin{thm}[Zaslavsky \cite{Zaslavsky1975}]
We have $\#\Ch(\A)=(-1)^d\chi_\A(-1)$.
In particular, $\#\Ch(\A)$ is determined by the intersection lattice $L(\A)$.
\end{thm}

The set $\Ch(\A)$ carries more structure than a mere set, namely, the adjacency relation among chambers.
The \emph{tope graph} of $\A$, denoted by $\T(\A)$, is the graph whose vertex set is $\Ch(\A)$ and whose edges join adjacent chambers.
This simple object contains essentially all the combinatorial information of the hyperplane arrangement:

\begin{thm}[Theorem 4.2.14 of \cite{Bjorner1999}]
The oriented matroid associated with $\A$ is uniquely determined by $\T(\A)$.
\end{thm}

Figure 1 of \cite{Yagi2026} gives a detailed account of which of the various invariants of hyperplane arrangements determine one another.
In particular, if two arrangements have isomorphic tope graphs, then their complex complements, intersection lattices, and filtered Varchenko--Gelfand algebras are isomorphic.

Our motivation is to refine the classical problem of counting chambers to an invariant that reflects the structure of the tope graph.
For this purpose, we use the notion of magnitude introduced by Leinster \cite{Leinster2013,Leinster2017}.
Magnitude is a cardinality-like invariant of enriched categories; in particular, it gives invariants of metric spaces and graphs.
The magnitude of a graph is a power series with integer coefficients whose constant term is equal to the number of vertices, while its higher-degree terms reflect the edge structure.
Thus, computing the magnitude of the tope graph may be regarded as a natural refinement of chamber counting.
Interestingly, the zeta matrix used to compute the magnitude of the tope graph coincides with the Varchenko matrix, which has been studied classically in the theory of hyperplane arrangements.

For a hyperplane arrangement $\A$, let $\Mag(\A,q)$ denote the magnitude of its tope graph.
We first show that $\Mag(\A,q)$ has the following structural properties:
\begin{itemize}
\item $\Mag(\A,q^{-1})=q^{\#\A}\Mag(\A,q)$.
\item $\Mag(\A,0)=\#\Ch(\A)$ and $\Mag(\A,1)=1$.
\item $\Mag(\A,q)$ can be written as $P(q)/Q(q)$ for palindromic polynomials $P,Q$. Moreover, every root of $Q$ is a root of unity.
\end{itemize}
Furthermore, if we define the interior magnitude by $\Mag^\circ (\A,q)=(-1)^{\rank \A}q^{\#\A}\Mag(\A,q)$, then the following \emph{face decomposition formula} holds:
\[
\Mag(\A,q)=\sum_{F\in \FF(\A)}\Mag^\circ(\A_F,q).
\]
Here, $\FF(\A)$ is the set of faces of $\A$, and $\A_F$ is the set of hyperplanes containing $F$.
For Coxeter arrangements, this recovers Steinberg's growth rate formula \cite{Steinberg} for finite Coxeter systems.
This formula provides an efficient recursive algorithm for computing $\Mag(\A,q)$.
Moreover, using Zaslavsky's theorem, we show that $\Mag(\A,q)$ is determined by the intersection lattice $L(\A)$.

We next carry out explicit computations for arrangements of rank $3$.
In this case, we prove that the signs of the coefficients of $\Mag(\A,q)$ eventually alternate.
Whether this mysterious phenomenon persists in higher rank remains an open problem.
Nevertheless, examples in which alternation fails in low degrees can indeed be constructed.

In the second half of the paper, we turn to the \emph{magnitude homology} of hyperplane arrangements.
Magnitude homology is a bigraded homology theory that categorifies magnitude; it was introduced for graphs by Hepworth--Willerton \cite{Hepworth2017} and for enriched categories by Leinster--Shulman \cite{Leinster2021}.
For a hyperplane arrangement $\A$, we denote the magnitude homology of its tope graph by $MH_{k,\ell}(\A)$.
We also define the interior magnitude homology $MH^\circ_{k,\ell}(\A)$, which categorifies interior magnitude, and prove the following homological version of the face decomposition formula:
\[
MH_{k,\ell}(\A)\cong \bigoplus_{F\in \FF(\A)}MH^\circ_{k,\ell}(\A_F).
\]
The proof uses the notion of conditional oriented matroids due to Bandelt--Chepoi--Knauer, the corner-peeling theorem of Knauer--Marc, and discrete Morse theory.
Unfortunately, since the relationship between $MH_{k,\ell}(\A)$ and $MH^\circ_{k,\ell}(\A)$ is unclear, this does not provide a recursive algorithm for computing $MH^\circ_{k,\ell}(\A)$.
We conjecture the following relationship between them:
\[
MH^\circ_{k,\ell}(\A)\cong MH_{k-\rank\A,\ell-\#\A}(\A).
\]
The face decomposition formula implies its Euler characteristic analogue, \(\chi_{\ell}^{\circ}(\A)=(-1)^{\rank \A}\chi_{\ell-\#\A}(\A)\).
This identity resembles combinatorial reciprocity theorems such as Ehrhart--Macdonald reciprocity \cite{Ehrhart1967Reciprocity,Macdonald1971} and Stanley's reciprocity theorem for posets \cite{Stanley1970}.

Another interesting aspect is the relationship between the magnitude homology of an arrangement and Booleanity.
We show that the following equivalence holds for every essential arrangement $\A$:
\[
\A\text{ is Boolean}\iff MH_{k,\ell}(\A)=0\text{ for }k\neq \ell.
\]
This suggests that the off-diagonal magnitude homology of an arrangement detects nontrivial dependencies in its matroid.
Conversely, we also show that if $\A$ is not a Boolean arrangement, then its diagonal interior magnitude homology vanishes.
These results allow us to determine the groups $MH_{k,\ell}(\A)$ completely when $k=\ell$.

The paper is organized as follows.
Section 2 recalls the necessary material from metric graph theory and the theory of hyperplane arrangements, establishing the foundational link between gated subgraphs and covector composition.
Section 3 investigates the basic properties and structure of the magnitude rational function.
Section 4 focuses on the sign pattern of magnitude power series.
Section 5 introduces the magnitude homology of conditional oriented matroids and proves the vanishing theorem needed in Section 6.
Section 6 explores magnitude homology of arrangements, presenting the Boolean diagonality theorem and the face decomposition theorem.
Finally, Section 7 poses several open problems and conjectures for future study. The appendix contains the results of several SageMath computer experiments.

\subsection*{Use of AI}
In writing this paper, we used AI in the following ways.
\begin{itemize}
    \item A substantial part of the proof strategy for Theorem \ref{facedecomp} was generated by ChatGPT-5.4 Pro.
    \item Part of the proof strategy for Theorems \ref{alternating} and \ref{thm:COM_vanish} was generated by ChatGPT-5.5 Pro.
    \item We used ChatGPT-5.5 Pro to proofread the manuscript for minor errors.
\end{itemize}
We emphasize, however, that we have not used any AI-generated text without careful verification. The text of this paper has been written under the full responsibility of the authors.

\subsection*{Convention}
Throughout this paper, all graphs are finite, simple, and undirected unless otherwise specified. 
When we discuss a topological property or construction of a poset, we mean that of its order complex.
We use the term ``arrangement'' to mean a real central hyperplane arrangement.

\section{Recollections}

\subsection{Metric graph theory}\label{basicdef} 
For a graph $G$, a \emph{walk} of length $k$ from a vertex $x\in G$ to a vertex $y\in G$ is a sequence of vertices $(x_0,\ldots,x_k)$ such that $x_0=x$, $x_k=y$, and $\{x_i,x_{i+1}\}\in E(G)$ for $i=0,1,\ldots,k-1$. We define a distance function $d:V(G)\times V(G)\to \Z\cup\{\infty\}$ by declaring $d(x,y)$ to be the minimum length among walks from $x$ to $y$. If there is no walk from $x$ to $y$, we set $d(x,y)=\infty$. Then $(V(G),d)$ is a generalized metric space. When $G$ is connected, $(V(G),d)$ is a metric space. We call the maximum value of $d$ the \emph{diameter} of $G$, denoted by $\diam(G)$.

A subgraph $H$ of $G$ can be equipped with a metric in two ways. One is the restriction of $d_G$ to $H$, and the other is the graph metric $d_H$ on $H$ itself. We say that $H$ is an \emph{isometric subgraph} of $G$ if the two metrics coincide, that is, if $d_H(x,y)=d_G(x,y)$ for $x,y\in H$.

In a graph $G$, a walk from $x$ to $y$ of minimal length $d(x,y)$ is called a \emph{minimal walk}. For $x,y\in G$, define the \emph{closed interval} $[x,y]_G$ as the set of all vertices traversed by a minimal walk from $x$ to $y$. We equip it with the partial order $u\prec_G v$ if there is a minimal walk from $x$ to $y$ that traverses $u$ before $v$. Note that $[x,y]_G$ has a unique minimal element $x$ and a unique maximal element $y$. A vertex $u$ belongs to $[x,y]_G$ if and only if $d(x,y)=d(x,u)+d(u,y)$. More generally, vertices $u_1,\ldots,u_k$ form a strictly linearly ordered chain in $[x,y]_G$, that is, $x\prec_G u_1\prec_G\cdots\prec_G u_k\prec_G y$, if and only if 
\begin{equation}
    d(x,y)=d(x,u_1)+d(u_1,u_2)+\cdots+d(u_{k-1},u_k)+d(u_k,y).\label{linearlyorderedchain}
\end{equation}
We define the \emph{open interval} $(x,y)_G$ as the subposet of $[x,y]_G$ obtained by excluding $x$ and $y$.

We say that a subgraph $H$ of $G$ is \emph{convex} in $G$ if, for any $x,y\in H$, every minimal walk from $x$ to $y$ traverses only vertices in $H$; equivalently, $[x,y]_G\subseteq H$ for $x,y\in H$. Every convex subgraph is an isometric subgraph, but the converse is false.

\begin{defn}[Gated subgraphs]\label{gated}
    Let $G$ be a connected graph and let $H$ be an isometric subgraph of $G$. We say that $H$ is \emph{gated} if for every $x\in G$ there is $u\in H$ such that, for any $v\in H$,
    \[
    d(x,v)=d(x,u)+d(u,v).
    \]
    Such a vertex $u$ must be unique, since the above equation says that $u$ is the closest vertex to $x$ in $H$. We call $u$ the \emph{gate} of $H$ with respect to $x$. Intuitively, every minimal walk from $x$ to a vertex of $H$ enters $H$ at the gate.
    We call the map $p:V(G)\to V(H)$ taking $x\in G$ to its gate in $H$ the \emph{gate projection}. Note that $p(x)=x$ for $x\in H$.
\end{defn}

\begin{prop}[{\cite[Proposition A.1]{Aguiar2017}}]\label{gatedisconvex}
    A gated subgraph is necessarily convex.
\end{prop}

In this paper, the graph homomorphisms we use are \emph{distance non-increasing} maps (1-Lipschitz maps) $f:G\to H$, that is, maps $f\colon V(G)\to V(H)$ satisfying $d_H(f(x),f(y))\leq d_G(x,y)$ for $x,y\in G$. Note that gate projections are distance non-increasing.

\subsection{Magnitude of graphs}\label{mag}
See Leinster \cite{Leinster2019} for details. 
For a finite simple graph $G$, define the zeta matrix $Z_G=Z_G(q)=[q^{d(x,y)}]_{x,y\in G}$, where $q^{\infty}=0$ by convention. 
Note that $Z_G(0)$ is the identity matrix and hence $\det Z_G(q)\in\Z[q]$ is a polynomial with constant term $1$. 
In particular, $Z_G(q)$ is invertible over $\Q(q)$, with inverse $Z_G(q)^{-1}=(\det Z_G(q))^{-1}\mathrm{adj}~Z_G(q)$, whose entries can be expressed as power series in $\Z[\![q]\!]$.

The \emph{magnitude} of $G$ is defined as the sum of all entries of $Z_G(q)^{-1}$:
\[
\Mag(G)=\Mag(G,q):=\sum_{x,y\in G}Z_G(q)^{-1}_{x,y}\in \Z[\![q]\!].
\]
For a graph $G$, define the weight vector $w_G$ as the unique solution to $Z_G\mathbf{x}=\mathbf{1}$.
Then we have
\[
\Mag(G)=\mathbf{1}^Tw_G=\sum_{x\in G}w_G(x).
\]

\begin{prop}[{\cite[Lemma 3.2]{Leinster2019}}]\label{vert-trans}
    If $G$ is vertex-transitive, then for any $g\in G$ we have
    \[
    \Mag(G)=\frac{\#V(G)}{\sum_{x\in G}q^{d(g,x)}}.
    \]
\end{prop}

\begin{proof}
    The denominator $s(g)=\sum_{x\in G}q^{d(g,x)}$ is independent of $g\in G$ by the assumption that $G$ is vertex-transitive, and we denote it by $s=s(g)$. 
    It is then easy to see that $w_G=[1/s \cdots 1/s]^T$ is the solution of $Z_G\mathbf{x}=\mathbf{1}$. 
    We conclude $\Mag(G)=\#V(G)/s$.
\end{proof}

\begin{ex}
    Let $K_n$ be the complete graph on $n$ vertices and let $C_n$ be the cycle graph on $n$ vertices. Both are vertex-transitive.
    Their magnitudes are given by
    \begin{align*}
        \Mag(K_n)&=\frac{n}{1+(n-1)q},\\
        \Mag(C_{2n})&=\frac{2n(q-1)}{(q^n-1)(q+1)}=\frac{2n}{(1+q)(1+q+q^2+\cdots+q^{n-1})},\\
        \Mag(C_{2n-1})&=\frac{(2n-1)(q-1)}{2q^n-q-1}=\frac{2n-1}{1+2q+2q^2+\cdots+2q^{n-1}}.
    \end{align*}
\end{ex}

Magnitude behaves like cardinality.
For two graphs $G$ and $H$, we have
\[
    \Mag(G\sqcup H)=\Mag(G)+\Mag(H),\quad \Mag(G\mathrel{\square} H)=\Mag(G)\Mag(H).
\]
Moreover, under certain conditions, magnitude satisfies inclusion-exclusion
\[
    \Mag(G\cup H)=\Mag(G)+\Mag(H)-\Mag(G\cap H).
\]
See \cite{Leinster2019} for details.

\subsection{Magnitude homology of graphs}
For a graph $G$, we call a tuple $\Vec{x}=(x_0,x_1,\ldots,x_k)\in V(G)^{k+1}$ a $k$-chain of $G$ and define its length by $\ell(\Vec{x})=d(x_0,x_1)+d(x_1,x_2)+\cdots+d(x_{k-1},x_k)$.
We say that $\Vec{x}$ is \emph{proper} if $x_i\neq x_{i+1}$ for all $i$. Let $P_{k}(G)$ be the set of proper $k$-chains of $G$, and let $P_{k,\ell}(G)$ be the set of proper $k$-chains of length $\ell$. The following result of Leinster is the starting point of the categorification of graph magnitude introduced by Hepworth--Willerton \cite{Hepworth2017}.

\begin{prop}[{\cite[Proposition 3.9]{Leinster2019}}]\label{coeff}
    For a graph $G$, write $\Mag(G)=\sum_{\ell=0}^{\infty}c_{\ell}q^{\ell}\in \Z[\![q]\!]$.
    Then
    \[
    c_{\ell}=\sum_{k=0}^{\ell}(-1)^k\#P_{k,\ell}(G).
    \]
    In particular, we have $c_0=\#V(G)$ and $c_1=-2\#E(G)$.
\end{prop}

Let $MC_{k,\ell}(G)$ be the free abelian group generated by $P_{k,\ell}(G)$. We define the boundary operator $\partial:MC_{k,\ell}(G)\to MC_{k-1,\ell}(G)$ by $\partial=\sum_{i=1}^{k-1}(-1)^{i-1}\partial_i$, where $\partial_i:MC_{k,\ell}(G)\to MC_{k-1,\ell}(G)$ is defined by
\[
    \partial_i(x_0,\ldots,x_k)=\begin{cases}
    (x_0,\ldots,\widehat{x_i},\ldots,x_k), & \text{ if } d(x_{i-1},x_{i+1})=d(x_{i-1},x_i)+d(x_i,x_{i+1});\\
    0, & \text{ otherwise},
\end{cases}
\]
where $\widehat{x_i}$ denotes deletion of $x_i$. One can verify that $\partial^2=0$.
We call the homology of the chain complex $(MC_{*,\ell}(G),\partial)$ the length $\ell$ \emph{magnitude homology} of $G$, denoted by $MH_{*,\ell}(G)$. Since the boundary map preserves endpoints, we have a decomposition
\[
MC_{k,\ell}(G)=\bigoplus_{u,v\in G}MC_{k,\ell}(G;u,v),
\]
where $MC_{k,\ell}(G;u,v)$ is the subgroup generated by $P_{k,\ell}(G;u,v)=\{\Vec{x}\in P_{k,\ell}(G)\mid x_0=u,\ x_k=v\}$. These subgroups form subcomplexes and induce a decomposition of homology groups
\begin{equation}
    MH_{k,\ell}(G)=\bigoplus_{u,v\in G}MH_{k,\ell}(G;u,v). \label{vertexdecomp}
\end{equation}
We define the magnitude Betti number by $\beta_{k,\ell}(G)=\rank MH_{k,\ell}(G)$ and the length $\ell$ magnitude Euler characteristic by $\chi_{\ell}(G)=\sum_k(-1)^k\beta_{k,\ell}(G)$.

\begin{prop}[{\cite[Propositions 2.9 and 2.10]{Hepworth2017}}]\label{lowdeg}
    Let $G$ be a graph. 
    \begin{enumerate}
        \item $MH_{0,0}(G)$ is free abelian of rank $\beta_{0,0}(G)=\#V(G)$.
        \item $MH_{1,1}(G)$ is free abelian of rank $\beta_{1,1}(G)=2\#E(G)$.
        \item If $MH_{k,\ell}(G)\neq 0$, then $\ell/\diam(G)\leq k\leq\ell$.
        The first inequality is strict if $\diam(G)>1$ and $\ell>0$.
    \end{enumerate}
\end{prop}

The following proposition is a straightforward consequence of Proposition \ref{coeff}.

\begin{prop}[{\cite[Theorem 2.8]{Hepworth2017}}]
    Let $G$ be a graph and write $\Mag(G)=\sum_{\ell=0}^{\infty}c_{\ell}q^{\ell}\in\Z[\![q]\!]$. Then
    \[
    c_{\ell}=\chi_{\ell}(G)=\sum_{k=0}^{\infty}(-1)^k\beta_{k,\ell}(G)=\sum_{k=\lceil \ell/\diam(G)\rceil}^{\ell}(-1)^k\beta_{k,\ell}(G).
    \]
\end{prop}

Hepworth--Willerton proved that magnitude homology is a functor:

\begin{prop}[{\cite[Proposition 3.3]{Hepworth2017}}]\label{functor}
    Magnitude homology is a functor from the category of finite graphs with distance non-increasing maps to the category of bigraded abelian groups.
\end{prop}

A direct application of functoriality is as follows. Let $G$ be a connected graph and let $H$ be a gated subgraph of $G$ with gate projection $p:G\to H$. Let $\iota:H\to G$ be the inclusion. Both $p$ and $\iota$ are distance non-increasing. The composition $p\circ\iota=\id_{H}$ induces $\id=p_*\circ\iota_*:MH_{k,\ell}(H)\to MH_{k,\ell}(G)\to MH_{k,\ell}(H)$. This gives a decomposition of $MH_{k,\ell}(G)$.
\begin{prop}\label{gateddecomp}
    If $p:G\to H$ is a gate projection, then
    \[
    MH_{k,\ell}(G)\cong MH_{k,\ell}(H)\oplus \ker(p_*).
    \]
\end{prop}

There is a K\"unneth formula for magnitude homology of graphs.

\begin{thm}[{\cite[Theorem 5.3]{Hepworth2017}}]\label{Kunneth}
There is a natural short exact sequence that splits non-naturally
\[
0\to MH_{*,*}(G)\otimes MH_{*,*}(H)\xrightarrow{\square} MH_{*,*}(G\mathrel{\square} H)\to \mathrm{Tor}(MH_{*-1,*}(G),MH_{*,*}(H))\to 0,
\]
where $\square$ is the exterior product defined in \cite[Definition 5.2]{Hepworth2017}.
\end{thm}

\begin{ex}\label{complete_graph}
    For the complete graph $K_n$, we have
    \[
    MH_{k,\ell}(K_n)=\begin{cases}
    \Z^{\oplus P_{\ell}(K_n)}, & \text{ if } k=\ell;\\
    0, & \text{ otherwise}.
    \end{cases}
    \]
The magnitude Betti number is given by $\beta_{\ell,\ell}(K_n)=n(n-1)^{\ell}$.
\end{ex}

\begin{ex}\label{hypercube}
    For the hypercube graph $Q_d=K_2^{\square d}$, the K\"unneth formula (Theorem \ref{Kunneth}) gives
    \[
    MH_{k,\ell}(Q_d)\cong\bigoplus_{\sum k_i=k, \sum \ell_i=\ell}MH_{k_1,\ell_1}(K_2)\otimes\cdots\otimes MH_{k_d,\ell_d}(K_2),
    \]
    Using the result for $K_2$, we obtain
    \[
    MH_{k,\ell}(Q_d)\cong\begin{cases}
    \Z^{\oplus r(d)}, & \text{ if } k=\ell;\\
    0, & \text{ otherwise},
\end{cases}
\]
where $r(d)=2^d\binom{\ell+d-1}{d-1}$.
\end{ex}

\begin{defn}[Diagonal graphs]
    We say that $G$ is \emph{diagonal} if $MH_{k,\ell}(G)=0$ for $k\neq\ell$.
\end{defn}

\begin{ex}\label{diag}
    We have seen in Examples \ref{complete_graph} and \ref{hypercube} that $K_n$ and $Q_d$ are diagonal. In particular, the cycle graph $C_4=Q_2$ is diagonal. However, Gu \cite{Gu2018} computed $MH_{k,\ell}(C_n)$ for all $n\geq 5$ and showed that the graphs $C_n$ ($n\geq 5$) are not diagonal.
\end{ex}

Kaneta--Yoshinaga \cite{Kaneta2021} constructed a computationally tractable approximation of magnitude homology.
Let $\Vec{x}=(x_0,\ldots,x_k)\in P_k(G)$. 
We say $\Vec{x}$ is \emph{geodesic} if $\ell(\Vec{x})=d(x_0,x_k)$. 
Every subchain of a geodesic proper chain is again geodesic and proper.
Any proper $1$-chain is geodesic.
A proper $2$-chain $(x_0,x_1,x_2)$ is geodesic if and only if $d(x_0,x_2)=d(x_0,x_1)+d(x_1,x_2)$.
If $u,v\in G$ are vertices and $\ell=d(u,v)$, then any chain in $P_{k,\ell}(G;u,v)$ must be geodesic. 

\begin{defn}[Geodesic magnitude homology]
We define the length $\ell$ \emph{geodesic magnitude complex} by
\[
MC_{*,\ell}^{\geod}(G):=\bigoplus_{\substack{u,v\in G\\ d(u,v)=\ell}}MC_{*,\ell}(G;u,v).
\]
Its homology is called the length $\ell$ \emph{geodesic magnitude homology}.
The geodesic magnitude homology $MH^{\geod}_{k,\ell}(G)$ is a direct summand of $MH_{k,\ell}(G)$.
It admits a decomposition
\[
MH_{k,\ell}^{\geod}(G)=\bigoplus_{\substack{u,v\in G\\ d(u,v)=\ell}}MH_{k,\ell}(G;u,v).
\]
Note that $MC_{*,\ell}^{\geod}(G)$ and $MH_{k,\ell}^{\geod}(G)$ are trivial if $\ell>\diam(G)$.
We define the geodesic magnitude Betti number by $\beta_{k,\ell}^{\geod}(G)=\rank MH_{k,\ell}^{\geod}(G)$ and the length $\ell$ geodesic magnitude Euler characteristic by $\chi_{\ell}^{\geod}(G)=\sum_{k}(-1)^k\beta_{k,\ell}^{\geod}(G)$. Observe that $\beta_{k,\ell}(G)\geq \beta_{k,\ell}^{\geod}(G)$.
\end{defn}

Geodesic magnitude homology can be computed using the order complex of open intervals.
Given $a,b\in G$, consider the open interval $(a,b)_G$ introduced in Section \ref{basicdef}.
We write $C_*(a,b)_G$ for the reduced chain complex of the order complex of $(a,b)_G$. 

\begin{prop}[{\cite[Proposition 2.3]{Gomi2026}}]\label{intervalMH}
    For $a,b\in G$ with $d(a,b)=\ell$, there is an isomorphism of chain complexes
    \[
    MC_{*,\ell}(G;a,b)\cong C_{*-2}(a,b)_G,
    \]
    which induces an isomorphism of homology groups
    \[
    MH_{k,\ell}(G;a,b)\cong\widetilde{H}_{k-2}(C_*(a,b)_G).
    \]
\end{prop}

\begin{proof}
    Any proper chain $\Vec{x}=(a,x_1,\ldots,x_{k-1},b)\in P_{k,\ell}(G;a,b)$ must be geodesic since $\ell=d(a,b)$.
    Mapping $\Vec{x}$ to the linearly ordered chain $x_1\prec_G\cdots\prec_G x_{k-1}$ in $C_{k-2}(a,b)_G$ defines the desired isomorphism.
\end{proof}

\begin{cor}\label{geodMHdecomp}
    The geodesic magnitude homology admits a decomposition
    \[
    MH_{k,\ell}^{\geod}(G)\cong\bigoplus_{\substack{a,b\in G\\ d(a,b)=\ell}}\widetilde{H}_{k-2}(C_*(a,b)_G).
    \]
\end{cor}

Geodesic magnitude homology is a coarse approximation of magnitude homology. As a more refined approximation, Kaneta--Yoshinaga \cite{Kaneta2021} introduced \emph{geodesically simple magnitude homology}.
For a proper chain $\Vec{x}=(x_0,\ldots,x_k)\in P_k(G)$, an internal node $x_i$ ($i=1,2,\ldots,k-1$) is called a \emph{smooth} node of $\Vec{x}$ if $d(x_{i-1},x_{i+1})=d(x_{i-1},x_i)+d(x_i,x_{i+1})$. A node that is not smooth is called a \emph{singular} node. Note that endpoints are singular. We write $\varphi(\Vec{x})$ for the subchain of $\Vec{x}$ consisting of all singular nodes and call it the frame of $\Vec{x}$. Note that $\varphi(\Vec{x})$ is not necessarily proper.

\begin{defn}[Geodesically simple magnitude homology]
    A proper chain $\Vec{x}$ is \emph{geodesically simple} if $\ell(\varphi(\Vec{x}))=\ell(\Vec{x})$.
    We define $MC_{k,\ell}^{\gs}(G)$ to be the subgroup of $MC_{k,\ell}(G)$ generated by geodesically simple chains.
    Then $MC_{*,\ell}^{\gs}(G)$ forms a subcomplex of $MC_{*,\ell}(G)$ \cite[Propositions 3.6 and 3.7]{Kaneta2021}.
    Its homology is called the \emph{geodesically simple magnitude homology} of $G$, denoted by $MH_{*,\ell}^{\gs}(G)$.
\end{defn}

How well $MH_{*,\ell}^{\gs}(G)$ approximates $MH_{*,\ell}(G)$ is controlled by the notion of a $4$-cut.

\begin{defn}[$4$-cuts]
    A proper $3$-chain $(x_0,x_1,x_2,x_3)\in P_3(G)$ is called a \emph{$4$-cut} of $G$ if $(x_0,x_1,x_2)$ and $(x_1,x_2,x_3)$ are geodesic but $(x_0,x_1,x_2,x_3)$ is not geodesic. Let $m_G\geq 3$ be the minimum length of $4$-cuts of $G$. We set $m_G=\infty$ if $G$ has no $4$-cuts.
    Any proper chain of length less than $m_G$ is geodesically simple.
\end{defn}

\begin{lem}\label{exist4cut}
    Suppose that $\Vec{x}=(x_0,\ldots,x_k)\in P_k(G)$ satisfies $\varphi(\Vec{x})=(x_0,x_k)$ and $\Vec{x}$ is not geodesically simple.
    That is, $(x_{i-1},x_i,x_{i+1})$ is geodesic for $i=1,\ldots,k-1$, but $\Vec{x}$ is not geodesic.
    Then $\Vec{x}$ has a subchain that is a $4$-cut.
\end{lem}

\begin{proof}
    Let $\Vec{y}=(y_0,\ldots,y_p)$ be a contiguous subchain of $\Vec{x}$, that is, $y_0=x_i,y_1=x_{i+1},\ldots,y_p=x_{i+p}$ for some $i$, such that $\Vec{y}$ is not geodesic and $p$ is as small as possible. This means $(y_0,\ldots,y_{p-1})$ and $(y_1,\ldots,y_p)$ are geodesic. Then $p\geq 3$. We claim $\Vec{z}:=(y_0,y_1,y_{p-1},y_p)$ is a $4$-cut. First note that $y_{p-1}\neq y_1$ as they are connected by a geodesic $(y_1,y_2,\ldots,y_{p-1})$. That $(y_0,y_1,y_{p-1})$ and $(y_1,y_{p-1},y_p)$ are geodesic follows from the assumption that $(y_0,\ldots,y_{p-1})$ and $(y_1,\ldots,y_p)$ are geodesic. Moreover, $\ell(\Vec{z})=d(y_0,y_1)+d(y_1,y_{p-1})+d(y_{p-1},y_p)=\sum_{i=1}^pd(y_{i-1},y_i)>d(y_0,y_p)$, since $\Vec{y}$ is not geodesic. This shows $\Vec{z}$ is not geodesic and hence is a $4$-cut.
\end{proof}

Let $\Vec{a}\in P_{m,\ell}(G)$ and $k\geq m$. We define $P_{k,\Vec{a}}(G):=\{\Vec{x}\in P_{k,\ell}(G)\mid \varphi(\Vec{x})=\Vec{a}\}$. Note that $P_{k,\Vec{a}}(G)$ consists of geodesically simple chains and is empty if $\varphi(\Vec{a})\neq \Vec{a}$. We further define $MC^{\gs}_{k,\Vec{a}}(G)$ as the free abelian group generated by $P_{k,\Vec{a}}(G)$, which is a subgroup of $MC_{k,\ell}^{\gs}(G)$. These groups form a subcomplex $MC^{\gs}_{*,\Vec{a}}(G)$ whose homology $MH^{\gs}_{*,\Vec{a}}(G)$ is called the \emph{$\Vec{a}$-framed magnitude homology} of $G$. Kaneta--Yoshinaga \cite{Kaneta2021} proved that $MH_{k,\ell}^{\gs}(G)$ decomposes as a direct sum of framed magnitude homology groups:
\begin{thm}[{\cite[Theorem 3.12]{Kaneta2021}}]\label{gsmh}
    Let $G$ be a graph and $k,\ell>0$. We have a decomposition
    \[
MH_{k,\ell}^{\gs}(G)\cong\bigoplus_{\substack{m\leq k\\ \Vec{a}\in P_{m,\ell}(G)}} MH^{\gs}_{k,\Vec{a}}(G).
\]
Furthermore, if $k=1$ or $\ell<m_G$, then $MH_{k,\ell}^{\gs}(G)\cong MH_{k,\ell}(G)$.
\end{thm}

For a frame $\Vec{a}=(a_0,a_1,\ldots,a_m)\in P_{m,\ell}(G)$, we denote by $C_*(\Vec{a})_G$ the total complex of $C_*(a_0,a_1)_G\otimes\cdots\otimes C_*(a_{m-1},a_m)_G$. Proposition \ref{intervalMH} can be generalized as follows.

\begin{thm}[{\cite[Theorem 4.4 and Corollary 4.5]{Kaneta2021}}]\label{fmh}
    Suppose that a frame $\Vec{a}$ has length $\ell(\Vec{a})<m_G$. Then there is an isomorphism of chain complexes
    \[
    MC^{\gs}_{*,\Vec{a}}(G)\cong C_{*-2m}(\Vec{a})_G,
    \]
    which induces isomorphisms of homology groups
    \[
    MH^{\gs}_{k,\Vec{a}}(G)\cong \widetilde{H}_{k-2m}(C_*(\Vec{a})_G).
    \]
\end{thm}

\begin{rem}
    The original statement of Theorem 4.4 in \cite{Kaneta2021} does not contain the assumption $\ell(\Vec{a})<m_G$.
    However, the result fails if a frame node $a_i$ of $\Vec{a}$ can be smoothed out by nodes from neighboring intervals. For example, let $G=C_4$ be the cycle graph of $4$ vertices, labeled by $0,1,2,3$ in cyclic order. Choose $\Vec{a}=(0,1,3)\in P_{2,3}(G)$. Then $(0,1)_G$ is empty and $(1,3)_G$ consists of two incomparable points $\{0,2\}$. Compare $MC^{\gs}_{3,\Vec{a}}(G)$ with $C_{-1}(\Vec{a})_G$. The element $\varnothing\otimes 2$ in $C_{-1}(\Vec{a})_G$ does not have a preimage in $MC^{\gs}_{3,\Vec{a}}(G)$, since the frame of $(0,1,2,3)$ is not equal to $\Vec{a}=(0,1,3)$.
    This is because $(0,1,2,3)$ is a $4$-cut. 
\end{rem}

\begin{proof}[Proof of Theorem \ref{fmh}]
    Let $\Vec{a}=(a_0,a_1,\ldots,a_m)\in P_{m,\ell}(G)$ such that $\varphi(\Vec{a})=\Vec{a}$. Then $\ell=\ell(\Vec{a})<m_G$. If $\Vec{x}\in P_{k,\Vec{a}}(G)$ is a basis element of $MC^{\gs}_{k,\Vec{a}}(G)$, then $\Vec{x}$ has the form
    \[
    \Vec{x}=(a_0,x^1_{1},\ldots,x^1_{k_1},a_1,x^2_{1},\ldots,x^2_{k_2},a_2,\ldots,a_m).
    \]
    Note that $k=m+\sum_{i=1}^m k_i$. The smooth nodes between $a_{i-1}$ and $a_i$ form a linearly ordered chain $\langle x^i\rangle:=(x^i_1\prec_G\cdots\prec_G x^i_{k_i})$ in $C_{k_i-1}(a_{i-1},a_i)_G$. Define $\Phi_k:MC^{\gs}_{k,\Vec{a}}(G)\to C_{k-2m}(\Vec{a})_G$ by 
    \[
    \Phi_k(\Vec{x})=\langle x^1\rangle\otimes\langle x^2\rangle\otimes\cdots\otimes\langle x^m\rangle.
    \]
    The map $\Phi_k$ is injective on basis elements. To prove that $\Phi_k$ is also surjective, consider a basis element $\langle y^1\rangle\otimes\cdots\otimes\langle y^m\rangle$ of $C_{k-2m}(\Vec{a})_G$, where $\langle y^i\rangle=(y^i_1\prec_G\cdots\prec_G y^i_{k_i})\in C_{k_i-1}(a_{i-1},a_i)_G$ for $i=1,\ldots,m$. Let
    \[
        \Vec{y}=(a_0,y^1_{1},\ldots,y^1_{k_1},a_1,y^2_{1},\ldots,y^2_{k_2},a_2,\ldots,a_m).
    \]
    We have $\Vec{y}\in P_{k,\ell}(G)$. It remains to prove $\varphi(\Vec{y})=\Vec{a}$. First note that $(a_{i-1},y^i_1,\ldots,y^i_{k_i},a_i)$ is geodesic and hence $y^i_{j}$ is smooth in $\Vec{y}$. Next suppose, for contradiction, that $a_i$ is smooth in $\Vec{y}$ for some $i=1,\ldots,m-1$, that is,
    \[
    d(y^i_{k_i},y^{i+1}_1)=d(y^i_{k_i},a_i)+d(a_i,y^{i+1}_1).
    \]
    Here $y^i_{k_i}$ is understood as $a_{i-1}$ if $k_i=0$. Similarly, $y^{i+1}_1$ is understood as $a_{i+1}$ if $k_{i+1}=0$. Note that $k_i$ and $k_{i+1}$ cannot both be zero, since otherwise $\varphi(\Vec{a})\neq\Vec{a}$. Consider the subchain
    \[
    \Vec{z}=(a_{i-1},y^i_1,\ldots,y^i_{k_i},a_i,y^{i+1}_1,\ldots,y^{i+1}_{k_{i+1}},a_{i+1}).
    \]
    It satisfies $\varphi(\Vec{z})=(a_{i-1},a_{i+1})$ and $\ell(\Vec{z})=d(a_{i-1},a_i)+d(a_i,a_{i+1})>d(a_{i-1},a_{i+1})$ by the assumption that $a_i$ is singular in $\Vec{a}$. By Lemma \ref{exist4cut}, $\Vec{z}$ contains a subchain $\Vec{w}$ that is a $4$-cut. However, $\ell(\Vec{w})\leq \ell(\Vec{z})\leq \ell(\Vec{y})=\ell$, which contradicts the assumption $\ell<m_G$. The verification that $\Phi_k$ defines a chain map is omitted.
\end{proof}

\subsection{Hyperplane arrangements}
The standard references are \cite{Orlik1992,Stanley2007}. Let $\A$ be a real central hyperplane arrangement in $\R^d$.
For two chambers $C,C^{\prime}\in\Ch(\A)$, we write $S(C,C^{\prime})$ for the set of hyperplanes in $\A$ that separate $C$ and $C^{\prime}$. 
The graph metric $d$ on the tope graph $\T(\A)$ satisfies $d(C,C^{\prime})=\#S(C,C^{\prime})$. 
Note that $\T(\A)$ is a \emph{partial cube}, that is, it can be embedded isometrically into the hypercube graph $Q_{\A}=\{+,-\}^{\A}$.

An element of the intersection lattice $L(\A)$ is called a \emph{flat}. Note that the maximal flat is $\hat{1}=\bigcap_{H\in\A}H=:\bigcap\A$, called the center of $\A$, and the minimal flat $\hat{0}$ is the ambient space $\R^d$. The arrangement is called \emph{essential} if the center is the origin.
The rank of $\A$ is defined as $\rank \A=\codim (\bigcap\A)$.
Using the M\"obius function $\mu:L(\A)\times L(\A)\to\Z$, we define the \emph{characteristic polynomial} of $\A$ by $\chi_\A(t)=\sum_{X\in L(\A)}\mu(\hat{0},X)t^{\dim X}$. Zaslavsky's theorem states that $\#\Ch(\A)=(-1)^d\chi_\A(-1)$ \cite{Zaslavsky1975}.

The arrangement $\A$ determines a stratification of $\R^d$; the strata are called \emph{faces}, or covectors. Let $\FF(\A)$ be the poset of faces, ordered by closure inclusion. The unique minimal face is the center $\bigcap\A$, and the maximal faces are the chambers $\Ch(\A)$. There is an order-reversing surjective map $s:\FF(\A)\to L(\A)$ taking a face $F$ to its support $s(F)$, that is, the minimal subspace of $\R^d$ containing $F$.

For a face $F\in\FF(\A)$, let $\A_F=\{H\in\A\mid F\subseteq H\}$. Also, for a flat $X\in L(\A)$, let $\A_X=\{H\in\A\mid X\subseteq H\}$ and $\A^X=\{H\cap X\mid H\in\A\setminus\A_X\}$. Note that if $F\in\FF(\A)$ has support $s(F)=X\in L(\A)$, then $\A_F=\A_X$. 
For a flat $X\in L(\A)$, we define
\[
c^X:=\#\Ch(\A^X)=\#s^{-1}(X),\quad c_X=\#\Ch(\A_X).
\]

For an arrangement $\A$ in $V=\R^d$, we fix a defining form $\alpha_H\in V^*$ for each $H\in\A$.
Then each face $F\in\FF(\A)$ can be represented by a sign vector $F=(F_H)_{H\in\A}\in\{+,-,0\}^{\A}$, where $F_H=\mathrm{sgn}(\alpha_H(F))$. Declaring $0<\pm$ determines a partial order on the set of sign vectors, which coincides with the face order of $\FF(\A)$.
The Minkowski sum $Z(\A)=\sum_{H\in\A}[-\alpha_H,\alpha_H]$ of the line segments in $V^*$ is called the \emph{zonotope} of $\A$. The poset of nonempty faces of $Z(\A)$ is anti-isomorphic to $\FF(\A)$. The $1$-skeleton of the zonotope $Z(\A)$ is isomorphic to the tope graph $\T(\A)$.

\begin{defn}[Tits product]
    For $F,G\in \{+,-,0\}^\A$, we define their \emph{Tits product} $FG\in \{+,-,0\}^\A$ by
    \[
    (FG)_H=\begin{cases}
    F_H, \text{ if }F_H\neq 0;\\
    G_H, \text{ if }F_H=0.
    \end{cases}
    \]
    Note that the Tits product is not commutative in general.
\end{defn}

For any $F,G\in\FF(\A)$, we have $FG\in \FF(\A)$.
Geometrically, $FG$ is the face reached by starting from a relative interior point of $F$ and moving a small distance toward a relative interior point of $G$ (see Figure \ref{fig:tits_product}).

\begin{figure}[htbp]
    \centering
    \begin{tikzpicture}[scale=0.7]
        
        % ==========================================
        % PANEL (a): Computing FG
        % ==========================================
        \begin{scope}[shift={(0,0)}]
            % Shade the resulting face FG (Sector between 0 and 60 degrees)
            \fill[blue!80!black!15] (0,0) -- (3.3,0) arc (0:60:3.3) -- cycle;
            
            % Draw the A2 arrangement lines
            \draw[thick] (-3.5,0) -- (3.5,0) node[right] {$H_1$};
            \draw[thick] (240:3.5) -- (60:3.5) node[above right] {$H_2$};
            \draw[thick] (300:3.5) -- (120:3.5) node[above left] {$H_3$};
            
            % Define points x in F and y in G
            \coordinate (x) at (2,0);          % F is the positive H1 ray
            \coordinate (y) at (150:2.5);      % G is the top-left chamber
            
            % Mark the faces F and G
            \node[blue!80!black, font=\bfseries, below=2pt] at (x) {$x \in F$};
            \fill[blue!80!black] (x) circle (2.5pt);
            
            \node[red!80!black, font=\bfseries, above=2pt] at (y) {$y \in G$};
            \fill[red!80!black] (y) circle (2.5pt);
            
            % Draw the trajectory from x to y
            \draw[->, thick, dashed, gray] (x) -- (y);
            
            % Draw the "small epsilon" vector leaving F and entering FG
            \draw[->, ultra thick, blue!80!black] (x) -- ++(-0.9, 0.28) node[above right, xshift=8pt, yshift=-4pt] {\small $\varepsilon$};
            
            % Label the resulting face
            \node[blue!80!black, font=\bfseries] at (30:2) {$FG$};
            
            % Subcaption
            \node at (0, -4) {(a) Computing $FG$};
        \end{scope}

        % ==========================================
        % PANEL (b): Computing GF
        % ==========================================
        \begin{scope}[shift={(8.5,0)}]
            % Shade the resulting face GF (which is just G, the 120 to 180 sector)
            \fill[red!80!black!15] (0,0) -- (120:3.3) arc (120:180:3.3) -- cycle;
            
            % Draw the A2 arrangement lines
            \draw[thick] (-3.5,0) -- (3.5,0) node[right] {$H_1$};
            \draw[thick] (240:3.5) -- (60:3.5) node[above right] {$H_2$};
            \draw[thick] (300:3.5) -- (120:3.5) node[above left] {$H_3$};
            
            % Define points x in F and y in G (Same as Panel A)
            \coordinate (x) at (2,0);
            \coordinate (y) at (150:2.5);
            
            % Mark the faces F and G
            \node[blue!80!black, font=\bfseries, below=2pt] at (x) {$x \in F$};
            \fill[blue!80!black] (x) circle (2.5pt);
            
            \node[red!80!black, font=\bfseries, above=2pt] at (y) {$y \in G$};
            \fill[red!80!black] (y) circle (2.5pt);
            
            % Draw the trajectory from y to x
            \draw[->, thick, dashed, gray] (y) -- (x);
            
            % Draw the "small epsilon" vector leaving G (but staying in G)
            \draw[->, ultra thick, red!80!black] (y) -- ++(0.9, -0.28) node[above, xshift=-10pt, yshift=3pt] {\small $\varepsilon$};
            
            % Label the resulting face
            \node[red!80!black, font=\bfseries] at (160:1.5) {$GF = G$};
            
            % Subcaption
            \node at (0, -4) {(b) Computing $GF$};
        \end{scope}

    \end{tikzpicture}
    \caption{A visualization of the Tits product in the $A_2$ arrangement. Let $F$ be a ray (blue) and $G$ be a chamber (red). \textbf{(a)} The product $FG$ represents taking a small step from $F$ toward $G$, which immediately lands in the adjacent chamber. \textbf{(b)} The product $GF$ represents taking a small step from $G$ toward $F$, so $GF = G$.}
    \label{fig:tits_product}
\end{figure}

\subsection{Graph-theoretic properties of tope graphs}

For a face $F\in\FF(\A)$, we let $\T(\A)_F$ be the induced subgraph of $\T(\A)$ with vertices $\Ch(\A)_F=\{C\in\Ch(\A)\mid F\leq C\}$.

\begin{lem}\label{facegated}
    For $F\in\FF(\A)$, the subgraph $\T(\A)_F$ of $\T(\A)$ is gated with gate projection $D\mapsto FD$. 
\end{lem}
\begin{proof}
    We must show that the equality
    \begin{equation}
        d(D,C)=d(D,FD)+d(FD,C) \label{gatedcond}
    \end{equation}
    holds for $D\in \Ch(\A)$ and $C\in\Ch(\A)_F$. In terms of sign vectors, we have $d(X,Y)=\sum_{H\in\A}d(X,Y)_H$, where
    \[
    d(X,Y)_H=\begin{cases}
        1,\text{ if }X_H\neq Y_H;\\
        0,\text{ if }X_H=Y_H.
    \end{cases}
    \]
    We prove that both sides of equation (\ref{gatedcond}) have the same contribution on each $H\in\A$. If $H\notin \A_F$, then $F_H\neq 0$ and $(FD)_H=F_H$; also, since $C\in\Ch(\A)_F$, we have $C_H=F_H$. Therefore (\ref{gatedcond}) evaluated at $H$ is $d(D,C)_H=d(D,FD)_H+0$, which is true since both sides equal $1$ if $D_H\neq F_H$ and $0$ if $D_H=F_H$. Next suppose $H\in \A_F$. Then $F_H=0$ and $d(D,FD)_H=0$. The equation (\ref{gatedcond}) evaluated at $H$ is $d(D,C)_H=0+d(FD,C)_H$, which is true since $D_H=(FD)_H$.
\end{proof}

We record some important examples of tope graphs.
\begin{ex}\label{topeboolean}
    Let $\A=Bl_d$ be the $d$-th Boolean arrangement, that is, the arrangement of $d$ coordinate hyperplanes in $\R^d$. Then the tope graph $\T(\A)$ is the whole hypercube graph $Q_{\A}=\{+,-\}^{\A}\cong Q_d$.
\end{ex}

\begin{ex}\label{topegeneric}
    Let $r\geq 2$. An arrangement of $n$ hyperplanes in $\R^r$ is said to be \emph{generic} if any subarrangement of $r$ hyperplanes has rank $r$.
    Generic arrangements are realizations of uniform matroids $U_{r,n}$. By abuse of notation, we use $U_{r,n}$ to denote a generic arrangement of $n$ hyperplanes in $\R^r$. For example, the tope graph of $U_{r,r+1}$ is isomorphic to the induced subgraph of $Q_{\A}=\{+,-\}^{\A}$ obtained by removing $(+,\dots,+)$ and $(-,\dots,-)$. 
\end{ex}

\begin{ex}\label{topecoxeter}
    Consider a finite Coxeter system $(W,S)$ and the associated Coxeter arrangement $\A_W$. The zonotope $Z(\A_W)$ is the $W$-permutahedron, and the tope graph $\T(\A_W)$ is isomorphic to the Cayley graph of $(W,S)$.
\end{ex}

Tope graphs have the following graph-theoretic properties. 
%Recall that an arrangement is \emph{simplicial} if all of its chambers are simplicial cones.
\begin{prop}[{\cite[Propositions 4.2.15 and 4.4.8]{Bjorner1999}}]\label{basic}
    Let $\A$ be an arrangement of $\rank \A=r$.
    The tope graph $\T=\T(\A)$ has the following properties:
    \begin{enumerate}
        \item For each vertex $C$, there is a unique vertex $-C$ such that $d(C,-C)=\diam(\T)$.
        \item The mapping $C\mapsto -C$ is a fixed-point-free automorphism of $\T$.
        \item $\T$ is bipartite.
        \item For each vertex $C$, we have $\deg(C)\geq r$, where equality holds if and only if $C$ is simplicial.
        \item $\T$ is $r$-connected.
    \end{enumerate}
\end{prop}

The homotopy type of open intervals in $\T(\A)$ is described as follows.

\begin{prop}[{\cite[Theorem 2.2]{Edelman1985}}, see also {\cite[Theorem 4.4.2]{Bjorner1999}}]\label{intervelhomotopy}
    Let $C_1,C_2\in \Ch(\A)$.
    Then the open interval $(C_1,C_2)_{\T(\A)}$ in $\T(\A)$ is homotopy equivalent to a sphere $S^{\rank \A_F-2}$ if $[C_1,C_2]_{\T(\A)}=\Ch(\A)_F$ for some $F\in\FF(\A)$, and is contractible otherwise.
\end{prop}

Varchenko introduced a bilinear form associated with an arrangement and computed its determinant \cite{Varchenko1993}. Generalizations to oriented matroids can be found in \cite{Hochstattler2019}. See \cite{Denham1999} and \cite[Sections 8.4--8.5]{Aguiar2017} for details. 
Here we only consider the equal-weight case. For an arrangement $\A$, we define the \emph{$q$-Varchenko matrix} by
\[
\mathcal{V}_q(\A)=[q^{d(C,D)}]_{C,D\in\Ch(\A)}.
\]
In other words, $\mathcal{V}_q(\A)$ is the zeta matrix of the tope graph $\T(\A)$ (see Section \ref{mag}).

\begin{thm}[\cite{Varchenko1993}, see also {\cite[Theorem 8.23]{Aguiar2017}}]\label{detvar}
    The determinant of the $q$-Varchenko matrix is
    \[
    \det\mathcal{V}_q(\A)=\prod_{X\in L(\A)\setminus\{\hat{0}\}}(1-q^{2\#\A_X})^{c^X\beta_X},
    \]
    where $\beta_X=(-1)^{\rank \A_X-1}\chi_{\A_X}'(1)$.
\end{thm}

For a fixed chamber $B$, we write $D_{\A,B}(q)=\sum_{C\in\Ch(\A)}q^{d(B,C)}$ for the distance generating function; it is also the sum of the entries in the $B$-row/column of the $q$-Varchenko matrix $\mathcal{V}_q(\A)$.

\begin{ex}\label{coxetergrowth}
    Let $(W,S)$ be a finite Coxeter system and $\A$ the corresponding Coxeter arrangement. For $w\in W$, the length $\ell_S(w)$ is defined as the minimal number $m$ such that $w$ can be written as $s_1\cdots s_m$ for $s_i\in S$. There is a unique longest element $w_0$ with $\ell_S(w_0)=\#\A$ and $\ell(w_0w)=\#\A-\ell(w)$ for $w\in W$ (see \cite{Humphreys1990}). The Poincar\'e polynomial (or growth series) of $W$ is
    \[
    W(q)=\sum_{w\in W}q^{\ell_S(w)}.
    \]
    Note that $W(q)$ is palindromic: $q^{\#\A}W(q^{-1})=W(q)$.
    For a subset $X\subseteq W$, we also let $X(q)=\sum_{w\in X}q^{\ell_S(w)}$. Since $\T(\A)$ is the Cayley graph of $(W,S)$ (Example \ref{topecoxeter}), we have $D_{\A,B}(q)=W(q)$ for every $B\in\Ch(\A)$.
\end{ex}

\begin{prop}[\cite{Solomon1966}]\label{Cox}
    If $\A$ is a Coxeter arrangement with exponents $e_1,\ldots,e_{\ell}$, then we have
    \[
    D_{\A,B}(q)=\prod_{i=1}^{\ell}[1+e_i]_q
    \]
    for every $B\in\Ch(\A)$, where $[n]_q=1+q+q^2+\cdots+q^{n-1}$ is the $q$-number.
\end{prop}

This decomposition formula has been generalized to supersolvable arrangements \cite{Bjorner1990} and inductively factored arrangements \cite{Jambu1995}. See also \cite{Abe2020,Oh2008,Giordani2026}.

\subsection{Conditional oriented matroids}
An oriented matroid abstracts the set of faces of a hyperplane arrangement. Here we review the more general notion of a conditional oriented matroid (COM), which abstracts the set of faces of a hyperplane arrangement in a fixed convex region \cite{BCK-COM}.

Fix a finite set $E$.
For two sign vectors $F,G\in \{+,-,0\}^E$, their \emph{separator} $S(F,G)$ is defined by
\[
S(F,G)=\bigl\{e\in E\mid \{F_e,G_e\}=\{+,-\}\bigr\}.
\]
A conditional oriented matroid (COM) is a set of sign vectors $\mathcal{L}\subset \{+,-,0\}^E$ with the following properties:
\begin{enumerate}
    \item For any $F,G\in \mathcal{L}$, the Tits product $F\cdot(-G)$ of $F$ and $-G$ also belongs to $\mathcal{L}$.
    \item For any $F,G\in \mathcal{L}$ and for any $e\in S(F,G)$, there exists $I\in \mathcal{L}$ such that $I_e=0$ and $I_{e'} = (FG)_{e'}$ for all $e'\in E\setminus S(F,G)$.
\end{enumerate}
The tope graph $\mathcal{T}(\mathcal{L})\subset Q_E$ of a COM $\mathcal{L}$ is the induced subgraph of $Q_E$ spanned by $\mathcal{L}\cap \{+,-\}^E$.

\begin{ex}[Realizable COM]
    Fix an ambient space $V=\mathbb{R}^d$.
    Let $D$ be an open convex subset of $V$ and let $\A$ be an arrangement of hyperplanes intersecting with $D$.
    We fix a defining form for each $H\in \A$.
    Then the set of possible sign patterns on $D$ defines a COM $\mathcal{L}\subset \{+,-,0\}^\A$.
    A COM of this form is called \emph{realizable}.
    A COM-graph is called realizable if the corresponding COM is realizable.
\end{ex}

In this paper, to simplify the discussion, we use the notion of a COM-graph, which is equivalent to the notion of a COM.
Let $Q_E=\{+,-\}^E$ be the hypercube graph.
A subgraph $G$ of $Q_E$ is called a \emph{partial cube} if $G$ is an isometric subgraph of $Q_E$.
Note that a nonempty partial cube is induced and connected.
A \emph{halfspace} of $G$ is a subgraph of the form $\{x\in G\mid x_e=\bullet\}$ for $e\in E$ and $\bullet\in \{+,-\}$.

\begin{prop}[{\cite[Theorem 2]{Albenque-Knauer}}]
    Let $G\subset Q_E$ be a partial cube and $H\subset G$ be a convex subgraph.
    Then $H$ can be written as an intersection of halfspaces.
    In particular, $H$ is also a partial cube.
\end{prop}

Let $G\subset Q_E$ be a partial cube and $X\subset G$ be a subset.
The \emph{support} of $X$ is defined by
\[
S(X)=\{e\in E\mid \text{there exist } x,y\in X\text{ with }x_e\neq y_e\}.
\]
An \emph{antipodal subgraph} of $G$ is a nonempty convex subgraph $H\subset G$ that is invariant under the automorphism of $Q_E$ that flips all coordinates in $S(H)$.

\begin{defn}[COM-graphs]
    A partial cube $G\subset Q_E$ is called a \emph{COM-graph} if all of its antipodal subgraphs are gated.
\end{defn}

\begin{thm}[{\cite[Theorem 1.1]{KM-COM}}]
    The tope graph of a COM is a COM-graph.
    This induces a bijection
    \[
    \bigl\{\text{COM }\mathcal{L}\subset \{+,-,0\}^E\bigr\}\xrightarrow{\cong}
    \{\text{COM-graph }G\subset Q_E\};\quad \mathcal{L}\mapsto \mathcal{T}(\mathcal{L}).
    \]
    The elements of $\mathcal{L}$ correspond bijectively to antipodal subgraphs of $\mathcal{T}(\mathcal{L})$.
\end{thm}

Knauer--Marc \cite{KM-corners} introduced a method called \emph{corner peeling}, which cuts off part of a COM-graph to produce a new COM-graph.
Let $G\subset Q_E$ be a COM-graph.
A \emph{corner} of $G$ is a nonempty subset $C$ of $G$ satisfying certain conditions; see \cite[Section 5]{KM-corners} for details.
Here we only need the following facts:
\begin{enumerate}
    \item Any antipodal subgraph of $G'=G\setminus C$ is convex in $G$.
    In particular, $G'$ is also a COM-graph.
    \item There is a unique maximal antipodal subgraph $H\subset G$ which contains $C$.
    \item No vertex of $G\setminus H$ is adjacent to a vertex of $C$.
\end{enumerate}

\begin{ex}
Let $G$ be the COM-graph induced by the vertices of $Q_3=\{+,-\}^3$ other than $(-,-,-)$. Geometrically, this is the realizable COM-graph corresponding to the non-central arrangement in $\mathbb{R}^2$ defined by the three lines $x=0$, $y=0$, and $x+y=1$. Then $C=\{(+,-,-)\}$ is a corner of $G$. Moreover, $H=\{(+,\pm,\pm)\}$ is the unique maximal antipodal subgraph containing $C$ (see Figure \ref{fig:com}).
\end{ex}

\begin{figure}[htpb]
\begin{tikzpicture}[
    scale=0.9,
    vertex/.style={circle, fill=black, inner sep=1.6pt},
    edge/.style={line width=0.7pt}
]

% Orthogonal projection along the (1,1,1)-direction.
% Coordinates are proportional to
%   X = x-y,   Y = (x+y-2z)/sqrt(3).

\coordinate (ppp) at (0,0);
\coordinate (ppm) at (-2,1.2);
\coordinate (pmp) at (2,1.2);
\coordinate (mpp) at (0,-2.4);
\coordinate (pmm) at (0,2.4);
\coordinate (mpm) at (-2,-1.2);
\coordinate (mmp) at (2,-1.2);

% The vertex mmm = (-,-,-) would project to (0,0), but is removed.

\node[blue!80!black] at (0,1.2) {$H=\{(+,\pm,\pm)\}$};
\node[red!80!black] at (0,2) {$C$};

% Edges of the cube, with (-,-,-) and its incident edges removed.
\draw[edge, very thick, color=blue!80!black] (ppp)--(ppm);
\draw[edge, very thick, color=blue!80!black] (ppp)--(pmp);
\draw[edge] (ppp)--(mpp);

\draw[edge, very thick, color=blue!80!black] (ppm)--(pmm);
\draw[edge] (ppm)--(mpm);

\draw[edge, very thick, color=blue!80!black] (pmp)--(pmm);
\draw[edge] (pmp)--(mmp);

\draw[edge] (mpp)--(mpm);
\draw[edge] (mpp)--(mmp);

% Vertices
\node[vertex] at (ppp) {};
\node[vertex] at (ppm) {};
\node[vertex] at (pmp) {};
\node[vertex] at (mpp) {};
\node[vertex, color=red!80!black] at (pmm) {};
\node[vertex] at (mpm) {};
\node[vertex] at (mmp) {};

% Labels
\node[above left] at (ppm) {$(+,+,-)$};
\node[above right] at (pmp) {$(+,-,+)$};
\node[below] at (mpp) {$(-,+,+)$};

\node[above] at (pmm) {$(+,-,-)$};
\node[below left] at (mpm) {$(-,+,-)$};
\node[below right] at (mmp) {$(-,-,+)$};

\node[below left=2pt] at (ppp) {$(+,+,+)$};

\end{tikzpicture}
\caption{
The COM-graph $Q_3\setminus\{(-,-,-)\}$ has a corner
$C=\{(+,-,-)\}$. The upper face $H=\{(+,\pm,\pm)\}$ is the unique
maximal antipodal subgraph containing $C$.
}
\label{fig:com}
\end{figure}

\begin{defn}[Corner peeling]
    Let $G\subset Q_E$ be a COM-graph.
    A \emph{corner peeling} of $G$ is a partition $G=C_1\sqcup\cdots\sqcup C_m$ such that $C_i$ is a corner of $G_i=G\setminus (C_1\sqcup\cdots\sqcup C_{i-1})$.
\end{defn}

\begin{thm}[{\cite[Proposition 5.5]{KM-corners}}]\label{thm:corner-peeling}
    Any realizable COM-graph admits a corner peeling.
\end{thm}

\subsection{Algebraic discrete Morse theory}

We shall use a standard consequence of algebraic discrete Morse theory.
Discrete Morse theory was introduced by Forman \cite{Forman} as a
combinatorial analogue of Morse theory for cell complexes. Its basic purpose is
to simplify a cell complex by pairing cells in
adjacent dimensions and cancelling them without changing the homotopy type.
The algebraic version, developed by Sköldberg and
independently by Jöllenbeck--Welker \cite{Skoldberg,JW}, formulates this
cancellation procedure directly for based chain complexes.

Let $C_*=(C_k,d_k)$ be a bounded chain complex of finitely generated free abelian
groups.
Choose a basis
$\mathcal B_k$ of $C_k$ for each $k$, and put
$\mathcal B=\bigcup_k \mathcal B_k$.
For $a\in \mathcal B_k$, write
\[
    d_k(a)=\sum_{b\in \mathcal B_{k-1}} [a:b]\,b,
    \quad [a:b]\in \mathbb Z .
\]
Associated with this based chain complex is a directed graph
$\Gamma(C_*,\mathcal B)$. Its vertices are the elements of $\mathcal B$, and
there is an arrow
$a\to b$
whenever $a\in \mathcal B_k$, $b\in \mathcal B_{k-1}$, and $[a:b]\neq 0$.

A \emph{matching} $M$ on $\Gamma(C_*,\mathcal B)$ is a set of arrows no two of
which share a vertex. Thus each basis element is incident to at most one
matched arrow. If $a\to b$ belongs to $M$, then the coefficient $[a:b]$ is
called the weight of the matched arrow.
A basis element that is not incident to any
arrow of $M$ is called \emph{critical}. We say that $M$ is \emph{complete} if it has no
critical vertices.

Let $\Gamma_M$ be the directed graph obtained from
$\Gamma(C_*,\mathcal B)$ by reversing every arrow in $M$ and leaving all
other arrows unchanged. A matching $M$ is called \emph{acyclic} if $\Gamma_M$ has
no directed cycles.
The algebraic Morse theorem says that, if $M$ is a Morse matching, then
$C_*$ is chain homotopy equivalent to the Morse complex generated by the
critical basis elements \cite[Theorem~1]{Skoldberg}. We shall only use the
following special case.

\begin{thm}\label{thm:complete-morse-matching}
Let $C_*$ be a bounded chain complex of finitely generated free abelian
groups, and let $\mathcal B_k$ be a chosen basis of
$C_k$ for each $k$. Suppose that $\Gamma(C_*,\mathcal B)$ admits a
complete acyclic matching $M$ such that the following condition holds:
\begin{itemize}
    \item if $a\in \mathcal B_k$ and $b\in \mathcal B_{k-1}$ are matched,
    then the coefficient $[a:b]$ is $\pm 1$.
\end{itemize}
Then $C_*$ is contractible. In particular,
$C_*$ is acyclic, that is, $H_k(C_*)=0$ for all $k\geq 0$.
\end{thm}

\section{Magnitude of arrangements}

For a real central hyperplane arrangement $\A$, we define the magnitude $\Mag(\A)=\Mag(\A,q)$ of $\A$ as the magnitude of its tope graph $\T(\A)$.

\subsection{Basic properties}
We first prove the symmetry of the magnitude with respect to $q\leftrightarrow q^{-1}$.

\begin{prop}\label{reci}
 Let $\A$ be an arrangement with $N=\#\A$ hyperplanes. Then
 \[
 \Mag(\A,q^{-1})=q^N\Mag(\A,q).
 \]
\end{prop}
\begin{proof}
Recall that there is an involution $\tau:\Ch(\A)\to\Ch(\A)$ sending $D$ to $-D$.
For any two chambers $C,D\in\Ch(\A)$, we have $d(C,D)+d(C,-D)=N$. Thus the $q$-Varchenko matrix $\mathcal{V}_q=\mathcal{V}_q(\A)=[q^{d(C,D)}]$ satisfies
\[
q^N\mathcal{V}_{q^{-1}}=[q^{N-d(C,D)}]=[q^{d(C,-D)}]=\mathcal{V}_{q}P_{\tau},
\]
where $P_{\tau}$ is the permutation matrix induced by $\tau$. By the definition of magnitude, we have
\begin{align*}
\Mag(\A,q^{-1})=\mathbf{1}^T(\mathcal{V}_{q^{-1}})^{-1}\mathbf{1}=\mathbf{1}^T(q^{-N}\mathcal{V}_qP_{\tau})^{-1}\mathbf{1}=
q^{N}\mathbf{1}^TP_{\tau}^{-1}\mathcal{V}_q^{-1}\mathbf{1}=q^{N}\mathbf{1}^T\mathcal{V}_q^{-1}\mathbf{1}=q^{N}\Mag(\A,q),  
\end{align*}
where we used $\mathbf{1}^TP_{\tau}^{-1}=\mathbf{1}^T$.
\end{proof}

Since tope graphs are partial cubes, $\Mag(\A)$ satisfies the one-point property established in \cite{Leinster2023}.
\begin{prop}\label{1pt}
    If $\A$ is an arrangement, then $\Mag(\A,1)=1$.
\end{prop}

Next we prove some structural results for $\Mag(\A)$.
\begin{thm}
For an arrangement $\A$, let $\Mag(\A,q)=P(q)/Q(q)$ be the reduced form of the magnitude, with the leading coefficient of $Q(q)$ positive. Then
\begin{enumerate}
    \item $\#\A=\deg Q-\deg P$.
    \item $P,Q$ are palindromic.
    \item All roots of $Q$ are roots of unity other than $1$.
\end{enumerate}
\end{thm}
\begin{proof}
Let $P(q)=a_mq^m+\cdots+a_0$ and $Q(q)=b_nq^n+\cdots+b_0$ where $a_m\neq 0$ and $b_n\neq 0$. Then $m=\deg P$ and $n=\deg Q$. Since $\Mag(\A,0)=a_0/b_0=\#\Ch(\A)$ (Proposition \ref{coeff}), we have $a_0\neq 0$ and $b_0\neq 0$. By the symmetry formula (Proposition \ref{reci}), we have
\[
q^N=\frac{\Mag(\A,q^{-1})}{\Mag(\A,q)}=\frac{P(q^{-1})Q(q)}{P(q)Q(q^{-1})}.
\]
Multiplying both sides by $q^{m-n}$, we obtain
\[
q^{N+m-n}=\frac{q^mP(q^{-1})Q(q)}{P(q)q^nQ(q^{-1})}=\frac{(a_m+\cdots+a_0q^m)(b_nq^n+\cdots+b_0)}{(a_mq^m+\cdots+a_0)(b_n+\cdots+b_0q^n)}.
\]
Comparing degrees, we obtain $N+m-n=0$, and (1) is proved.

Write $P^*(q)=q^mP(q^{-1})=a_m+\cdots+a_0q^m$ and $Q^*(q)=q^nQ(q^{-1})=b_n+\cdots+b_0q^n$ for the reverse polynomials. We have proved
\[
\Mag(\A,q)=\frac{P(q)}{Q(q)}=\frac{P^*(q)}{Q^*(q)}.
\]
Note that the fraction $P^*(q)/Q^*(q)$ is reduced, since if $P^*$ and $Q^*$ share a common root $\alpha\neq0$, then $P$ and $Q$ would share the common root $1/\alpha$, contradicting the assumption that $P/Q$ is reduced. Hence we must have $P^*(q)=rP(q)$ and $Q^*(q)=rQ(q)$ for some constant $r\neq0$. Substituting $q=1$ into the equations $Q^*(q)=q^nQ(1/q)$ and $Q^*(q)=rQ(q)$, we obtain
\[
Q^*(1)=Q(1)=rQ(1).
\]
Since $\Mag(\A,1)=P(1)/Q(1)=1$ (Proposition \ref{1pt}) and $P/Q$ is reduced, we have $Q(1)\neq 0$, so $r=1$. Then $P^*=P$ and $Q^*=Q$, and (2) is proved.

Since $\mathcal{V}_q^{-1}=\frac{1}{\det\mathcal{V}_q}\mathrm{adj~}\mathcal{V}_q$, the roots of $Q(q)$ are roots of $\det\mathcal{V}_q$, which are roots of unity by Theorem \ref{detvar}. We have already seen that $Q(1)\neq 0$. Thus (3) is proved.
\end{proof}

If an arrangement $\A$ decomposes as a direct sum $\A_1\oplus\A_2$, then the tope graph decomposes as the Cartesian product $\T(\A)=\T(\A_1)\square\T(\A_2)$. Therefore, by Proposition \ref{magprod}, the magnitude decomposes multiplicatively.
\begin{prop}
    $\Mag(\A_1\oplus\A_2)=\Mag(\A_1)\Mag(\A_2)$.
\end{prop}

Coxeter arrangements carry a natural action of the corresponding Coxeter groups on the chambers, which makes their magnitudes easy to compute.
\begin{prop}\label{coxeter}
    For a Coxeter arrangement $\A_W$ with exponents $\{e_1,\ldots,e_{\ell}\}$, we have
    \[
    \Mag(\A_W)=\frac{\#W}{W(q)}=\frac{\#W}{\prod_i[1+e_i]_q}.
    \]
\end{prop}
\begin{proof}
    The tope graph $\T(\A_W)$ is the $1$-skeleton of the $W$-permutahedron, or equivalently the Cayley graph of $W$ with respect to the standard generating set (Example \ref{topecoxeter}), and hence is vertex-transitive. The result follows from Proposition \ref{vert-trans} and Proposition \ref{Cox}.
\end{proof}

\subsection{Face decomposition formula}

For use in reciprocity statements, we make the following definition.
\begin{defn}\label{def:intmag}
    We define the \emph{interior magnitude} of $\A$ by
    \[
    \Mag^{\circ}(\A,q):=(-1)^{\rank \A}\Mag(\A,q^{-1})=(-1)^{\rank \A}q^{\#\A}\Mag(\A,q),
    \]
    where the latter equality follows from Proposition \ref{reci}.
\end{defn}

We prove the following ``face decomposition formula'' for the magnitude of an arrangement.

\begin{thm}\label{facedecomp}
    For any arrangement $\A$, we have
    \[
           \Mag(\A)=\sum_{F\in\FF(\A)}\Mag^{\circ}(\A_F)=\sum_{F\in\FF(\A)}(-1)^{\codim F}q^{\#\A_F}\Mag(\A_F).
    \]
\end{thm}

\begin{rem}
    Theorem \ref{facedecomp} can be reformulated as
    \[
    \Mag(\A)=\sum_{X\in L(\A)}c^X\Mag^{\circ}(\A_X)=\sum_{X\in L(\A)}(-1)^{\codim X}c^Xq^{\#\A_X}\Mag(\A_X),
    \]
    where $c^X=\#\Ch(\A^X)$.
    It can also be written as
    \[
    \Mag(\A)=\frac{\sum_{F\in\FF(\A)\setminus\{\cap\A\}}\Mag^{\circ}(\A_F)}{1-(-1)^{\rank \A}q^{\#\A}}=\frac{\sum_{F\in\FF(\A)\setminus\{\cap\A\}}(-1)^{\codim F}q^{\#\A_F}\Mag(\A_F)}{1-(-1)^{\rank \A}q^{\#\A}}.
    \]
    This can be regarded as a recursive formula for computing $\Mag(\A)$.
\end{rem}

\begin{rem}
    Taking the limit $q\to 1$ gives $\Mag(\A,1)=1$ and $\Mag^\circ(\A,1)=(-1)^{\rank \A}$ by Proposition \ref{1pt}.
    Therefore, Theorem \ref{facedecomp} can be regarded as a $q$-analogue of the decomposition of the Euler characteristic of the zonotope into its open faces.
\end{rem}

\begin{cor}
     The magnitude $\Mag(\A)$ of an arrangement $\A$ is determined by $L(\A)$.
\end{cor}

\begin{rem}\label{steinberg}
    Let $\A_W$ be the Coxeter arrangement corresponding to $(W,S)$, we know $\Mag(\A_W)=\#W/W(q)$ (Example \ref{coxeter}). The faces in $\FF(\A_W)$ correspond bijectively to cosets $wW_I$ of parabolic subgroups $W_I$ for $I\subseteq S$. The number of faces $F$ corresponding to $I$ is $\#W/\#W_I$. For such a face $F$, $\A_F$ is isomorphic to the Coxeter arrangement $\A_{W_I}$, with rank $\#I$, and interior magnitude
    \[
    \Mag^{\circ}(\A_F)=(-1)^{\#I}q^{\#\A_{W_I}}\Mag(\A_F)=(-1)^{\#I}\frac{\#W_I}{q^{-\#\A_{W_I}}W_I(q)}=(-1)^{\#I}\frac{\#W_I}{W_I(q^{-1})},
    \]
    where the last equality is by palindromicity of $W_I(q)$ (Example \ref{coxetergrowth}). Then Theorem \ref{facedecomp} applied to this case yields
    \[
    \frac{\#W}{W(q)}=\sum_{I\subseteq S}\frac{\#W}{\#W_I}\cdot(-1)^{\#I}\frac{\#W_I}{W_I(q^{-1})}=\sum_{I\subseteq S}(-1)^{\#I}\frac{\#W}{W_I(q^{-1})}.
    \]
    This recovers Steinberg's growth rate formula \cite{Steinberg} in the case of finite Coxeter groups.
\end{rem}

\begin{lem}\label{key_identity}
For any arrangement $\A$ and $D\in \Ch(\A)$, we have
\[
\sum_{F\in \FF(\A)}(-1)^{\codim F}q^{\#\A_F+d(D,FD)}=1.
\]
\end{lem}

\begin{proof}
If $\rank \A=0$, then $\A$ is the empty arrangement and there is nothing to prove, so let us assume $r:=\rank \A\ge 1$. Passing to the essentialization does not change the chamber set, the tope graph, or the quantities $\#\A_F$ and $\codim F=\rank(\A_F)$, so we may work in an essential arrangement of rank $r$.

Choose defining forms $\alpha_H$ so that $D_H=+$ for every $H\in \A$.
For a face $F$, define $M_D(F)\subset \A$ by
\[
M_D(F):=\{H\in \A: F_H\neq +\}.
\]
Since $F_H=0$ exactly when $H\in \A_F$, while $d(D,FD)$ counts the hyperplanes on which $F_H=-$, we have $\#M_D(F)=\#\A_F+d(D,FD)$.
Therefore it is enough to prove
\[
S_D(q):=\sum_{F\in \FF(\A)}(-1)^{\codim F}q^{\#M_D(F)}=1.
\]
Expand each power as
\[
q^{\#M_D(F)}=(1+(q-1))^{\#M_D(F)}
=\sum_{T\subseteq M_D(F)}(q-1)^{\#T}.
\]
After exchanging the order of summation, we obtain
\[
S_D(q)=\sum_{T\subseteq \A}(q-1)^{\#T}E_T,
\qquad
E_T:=\sum_{F:\,T\subseteq M_D(F)}(-1)^{\codim F}.
\]
Hence the desired identity will follow once we show $E_T=\mathbf{1}_{T=\varnothing}$, i.e. $E_T=0$ if $T\neq\varnothing$ and $E_{\varnothing}=1$.
We interpret $E_T$ as a reduced Euler characteristic. Let $S^{r-1}$ be the unit sphere in the ambient space $V=\R^r$. For each hyperplane $H\in \A$, let
\[
H_D^-:=\{x\in S^{r-1}:\alpha_H(x)\le 0\}
\]
be the closed hemisphere opposite to $D$, and for $T\subseteq \A$ define
\[
\Sigma_T:=\bigcap_{H\in T}H_D^- \subseteq S^{r-1},
\qquad
\Sigma_{\varnothing}:=S^{r-1}.
\]
The arrangement induces a regular CW decomposition of $S^{r-1}$ whose open cells are $F\cap S^{r-1}$ for non-minimal faces $F\in \FF(\A)\setminus\{\cap \A\}$.
An open cell $F\cap S^{r-1}$ lies in $\Sigma_T$ if and only if $T\subseteq M_D(F)$. Therefore
\[
\chi(\Sigma_T)
=\sum_{\substack{F\in \FF(\A)\setminus\{\cap \A\}\\ T\subseteq M_D(F)}}(-1)^{\dim F-1}.
\]
Since $\codim F=r-\dim F$ and the minimal face $\cap \A$ contributes $(-1)^r$ to $E_T$, we obtain
\begin{align*}
E_T
&=(-1)^r+\sum_{\substack{F\in \FF(\A)\setminus\{\cap \A\}\\ T\subseteq M_D(F)}}(-1)^{r-\dim F}
=(-1)^{r-1}\left(\chi(\Sigma_T)-1\right).
\end{align*}
If $T=\varnothing$, then $\Sigma_{\varnothing}=S^{r-1}$, so $E_{\varnothing}=(-1)^{r-1}(\chi(S^{r-1})-1)=1$.
If $T\neq\varnothing$, then $\Sigma_T$ is the intersection of a nonempty collection of closed hemispheres.
Moreover, $\Sigma_T$ contains $-D$, so it has nonempty interior.
This shows that $\Sigma_T$ is homeomorphic to $D^{r-1}$ and hence
$E_T=(-1)^{r-1}(\chi(D^{r-1})-1)=0$.
This completes the proof that $E_T=\mathbf{1}_{T=\varnothing}$.
\end{proof}

\begin{proof}[Proof of Theorem \ref{facedecomp}]
Fix a face $F\in \FF(\A)$.
Let $w_F\colon \Ch(\A)_F\to \Q(q)$ be the weight vector of $\T(\A)_F\cong \T(\A_F)$, so that
\[
\sum_{C\in \Ch(\A)_F} q^{d(E,C)}w_F(C)=1
\]
for every $E\in \Ch(\A)_F$.
Extend $w_F$ by zero outside $\Ch(\A)_F$.
Now define a function $\widetilde w_\A\colon \Ch(\A)\to \Q(q)$ by
\[
\widetilde w_\A(C)
:=\sum_{F\in \FF(\A)}(-1)^{\codim F}q^{\#\A_F}w_F(C).
\]
Taking the sum over all chambers gives
\begin{align*}
\sum_{C\in \Ch(\A)}\widetilde w_\A(C)
&=\sum_{F\in \FF(\A)}(-1)^{\codim F}q^{\#\A_F}
\sum_{C\in \Ch(\A)_F}w_F(C)\\
&=\sum_{F\in \FF(\A)}(-1)^{\codim F}q^{\#\A_F}\Mag(\A_F,q)\\
&=\sum_{F\in \FF(\A)}\Mag^{\circ}(\A_F,q).
\end{align*}
Therefore Theorem \ref{facedecomp} will follow once we prove that $\widetilde w_\A$ is the actual weight vector of $\A$.
Fix a chamber $D\in \Ch(\A)$. We compute the $D$-th coordinate of $\mathcal{V}_q(\A)\widetilde w_\A$:
\begin{align*}
\sum_{C\in \Ch(\A)} q^{d(D,C)}\widetilde w_\A(C)
&=\sum_{F\in \FF(\A)}(-1)^{\codim F}q^{\#\A_F}
\sum_{C\in \Ch(\A)_F} q^{d(D,C)}w_F(C).
\end{align*}
Since $\T(\A)_F$ is gated with gate projection $D\mapsto FD$ (Lemma \ref{facegated}), we have
\[
d(D,C)=d(D,FD)+d(FD,C)
\]
for every $C\in \Ch(\A)_F$.
Using this identity and the defining equation for the local weight vector $w_F$, we obtain
\begin{align*}
\sum_{C\in \Ch(\A)} q^{d(D,C)}\widetilde w_\A(C)
&=\sum_{F\in \FF(\A)}(-1)^{\codim F}q^{\#\A_F+d(D,FD)}
\sum_{C\in \Ch(\A)_F} q^{d(FD,C)}w_F(C)\\
&=\sum_{F\in \FF(\A)}(-1)^{\codim F}q^{\#\A_F+d(D,FD)}.
\end{align*}
By Lemma \ref{key_identity}, this equals $1$ for every $D\in \Ch(\A)$.
\end{proof}

\subsection{Examples} \label{firstexamples}

We use the notation $\A(n,k)$ from Gr\"unbaum's catalogue \cite{Grunbaum2009}.

\begin{ex}\label{nearpencil}
    The near-pencil $\A(n,0)$ has the tope graph $C_{2(n-1)}\Box K_2$ and hence
    \begin{align*}
        \Mag(\A(n,0))=\frac{4(n-1)}{[2]_q^2[n-1]_q}.
    \end{align*}
\end{ex}

\begin{ex}\label{Coxeterrank3}
Note that $\A(6,1)$ is the Coxeter arrangement of type $A_3$ or $D_3$, and $\A(9,1)$ is the Coxeter arrangement of type $B_3$.
Therefore, the magnitudes of both arrangements can be computed from Proposition \ref{coxeter}. In what follows, $\Phi_k(q)\in\Z[q]$ is the $k$-th cyclotomic polynomial.
\begin{align*}
    \Mag(\A(6,1))&=\frac{24}{[4]_q!}=\frac{24}{\Phi_2(q)^2\Phi_3(q)\Phi_4(q)}\\
    &=24-72q+96q^2-72q^3+48q^4-72q^5+120q^6-144q^7+144q^8-144q^9+144q^{10}+\cdots
\end{align*}
\begin{align*}
    \Mag(\A(9,1))&=\frac{48}{[2]_q[4]_q[6]_q}=\frac{48}{\Phi_2(q)^3\Phi_3(q)\Phi_4(q)\Phi_6(q)}\\
    &=48-144q+192q^2-192q^3+240q^4-336q^5+432q^6-528q^7+624q^8-720q^9+816q^{10}+\cdots
\end{align*}
Additional examples are given in the appendix.
\end{ex}

\section{Computations for rank 3 arrangements}

\subsection{Magnitude of rank 3 arrangements}
Using the face decomposition formula (Theorem \ref{facedecomp}), we can explicitly compute the magnitude of arrangements of rank $3$.

Let $\A$ be an arrangement of rank $3$ with $n=\#\A$.
For $k\geq 2$, we write $n_k$ for the number of faces $F\in \FF(\A)$ with $\codim F=2$ and $\#\A_F=k$.

\begin{thm}\label{thm:rank3}
    Let $\A$ be an arrangement of rank $3$ with $\#\A=n$.
    Then we have
    \[
    \Mag(\A)=\dfrac{1}{1+q^n}\biggl(\#\Ch(\A)-\sum_{k\geq 2}\dfrac{2kn_kq(1-q^{k-1})}{(1+q)(1-q^k)}\biggr).
    \]
\end{thm}

\begin{proof}
    Let $\FF_d(\A)$ denote the set of faces of codimension $d$ of $\A$.
    By Theorem \ref{facedecomp}, we have
    \begin{align*}
    \Mag(\A)&{}=\sum_{F\in \FF_0(\A)}1-\sum_{F\in \FF_1(\A)}\dfrac{2q}{1+q}+\sum_{F\in \FF_2(\A)}\dfrac{2\#\A_Fq^{\#\A_F}(1-q)}{(1+q)(1-q^{\#\A_F})}-q^n\Mag(\A,q)\\
    &{}=\#\Ch(\A)-\#\FF_1(\A)\cdot \dfrac{2q}{1+q}+\sum_{k=2}^\infty\dfrac{2kn_kq^k(1-q)}{(1+q)(1-q^k)}-q^n\Mag(\A,q).
    \end{align*}
    For each $F\in \FF_2(\A)$, the number of faces $G\in \FF_1(\A)$ of codimension $1$ adjacent to $F$ is exactly $2\#\A_F$.
    Conversely, for each $G\in \FF_1(\A)$, there are exactly two faces $F\in \FF_2(\A)$ of codimension $2$ adjacent to $G$.
    Therefore we have
    \[
    \#\FF_1(\A)=\dfrac{1}{2}\sum_{F\in \FF_2(\A)}2\#\A_F=\sum_{k=2}^\infty kn_k.
    \]
    Using this identity, we obtain
    \begin{align*}
    (1+q^n)\Mag(\A)&{}=\#\Ch(\A)-\sum_{k=2}^\infty\dfrac{2kn_kq}{1+q}+\sum_{k\geq 2}\dfrac{2kn_kq^k(1-q)}{(1+q)(1-q^k)}\\
    &{}=\#\Ch(\A)-\sum_{k=2}^\infty\dfrac{2kn_kq(1-q^{k-1})}{(1+q)(1-q^k)}.
    \end{align*}
    This completes the proof.
\end{proof}

\begin{ex}
    The arrangement $\A(6,1)$ from Gr\"unbaum's catalogue \cite{Grunbaum2009} satisfies
    \[
    \#\Ch(\A(6,1))=24\quad\text{and}\quad
    n_k=\begin{cases}
        6,&\text{ if }k=2;\\
        8,&\text{ if }k=3;\\
        0,&\text{ if }k\geq 4.
    \end{cases}
    \]
    By Theorem \ref{thm:rank3}, its magnitude is given by
    \begin{align*}
        \Mag(\A(6,1))&{}=\dfrac{1}{1+q^6}\biggl(24-\dfrac{24q}{(1+q)^2}-\dfrac{48q}{1+q+q^2}\biggr)\\
        &{}=\frac{24}{(1+q)^2(1+q+q^2)(1+q^2)}.
    \end{align*}
    This coincides with the formula given in Example \ref{Coxeterrank3}.
\end{ex}

\subsection{Alternating sign property}
In all examples shown in Appendix \ref{appendix}, the coefficients in the series expansion of the magnitude of each arrangement seem to have alternating signs. That is, if we write
\[
\Mag(\A)=\sum_{\ell=0}^\infty c_\ell q^\ell,
\]
then $(-1)^\ell c_\ell\geq 0$ appears to hold.
Here we prove that, if $\rank \A\leq 3$, then $(-1)^\ell c_\ell\geq 0$ holds for all sufficiently large $\ell$.
On the other hand, in the next subsection, we construct an example for which the inequality
$(-1)^\ell c_\ell \geq 0$
fails for small $\ell$.

\begin{defn}
    For each $k\geq 2$, we define a rational function $G_k(q)$ and a sequence $(g_{k,\ell})_{\ell\geq 0}$ by
    \[
    G_k(q)=\dfrac{q(1-(-q)^{k-1})}{(1-q)(1-(-q)^k)}=\sum_{\ell=0}^\infty g_{k,\ell}q^\ell.
    \]
    Theorem \ref{thm:rank3} implies that, for an arrangement $\A$ of rank $3$ with $\#\A=n$, we have
    \begin{align}\label{mag_minus_q}
    \Mag(\A,-q)=\dfrac{1}{1+(-q)^n}\biggl(\#\Ch(\A)+\sum_{k=2}^\infty 2kn_kG_k(q)\biggr).
    \end{align}
\end{defn}

\begin{lem}\label{n2positive}
    Let $\A$ be an arrangement of rank $3$.
    Let $n_k$ be the number of faces $F\in \FF(\A)$ with $\codim F=2$ and $\#\A_F=k$.
    Then we have $n_2>0$.
\end{lem}

\begin{proof}
    We may assume that $\A$ is essential.
    Consider the zonotope $Z(\A)$ of $\A$, which is a polytope in $\R^3$.
    The facets of $Z(\A)$ are polygons with an even number of vertices.
    The number of facets with $2k$ vertices is exactly $n_k$.
    If $n_2=0$, then every facet of $Z(\A)$ must have at least $6$ vertices.
    However, Euler's polyhedral formula shows that this is impossible.
    Therefore we have $n_2>0$.
\end{proof}

\begin{thm}\label{alternating}
    Let $\A$ be an arrangement with $\rank \A\leq 3$.
    Let $c_\ell$ denote the coefficient of $q^\ell$ in the series expansion of $\Mag(\A)$:
    \[
    \Mag(\A)=\sum_{\ell=0}^\infty c_\ell q^\ell.
    \]
    Then we have $(-1)^\ell c_\ell\geq 0$ for all sufficiently large $\ell$.
\end{thm}

\begin{proof}
When $\rank\A\leq 2$, the claim can be checked directly.
Therefore, we may assume that $\rank \A=3$.
Let $\#\A=n$.
Write $n_k$ for the number of faces $F\in \FF(\A)$ with $\codim F=2$ and $\#\A_F=k$.
Define $(e_\ell)_{\ell\geq 0}$ by
\[
\Mag(\A,-q)=\dfrac{1}{1+(-q)^n}\biggl(\#\Ch(\A)+\sum_{k=2}^\infty 2kn_kG_k(q)\biggr)=\sum_{\ell=0}^\infty e_\ell q^\ell.
\]
It suffices to show that $e_\ell>0$ for all sufficiently large $\ell$.
Put
\[
H(q):=\#\Ch(\A)+\sum_{k=2}^{\infty}2k n_kG_k(q),\qquad
H(q)=\sum_{\ell=0}^\infty a_\ell q^\ell .
\]
We first record two elementary facts about the coefficients $g_{k,\ell}$ of $G_k$.
If $k$ is even, then
\[
G_k(q)=\frac{q(1+q^{k-1})}{(1-q)(1-q^k)}
      =\frac{q}{1-q}\sum_{j=0}^\infty\left(q^{kj}+q^{kj+k-1}\right).
\]
Hence $g_{k,\ell}$ is non-negative and non-decreasing in $\ell$.  In particular,
for $k=2$ we have $g_{2,\ell}=\ell$.
If $k$ is odd, then
\[
G_k(q)=\frac{q(1-q^{k-1})}{(1-q)(1+q^k)}
      =\frac{q(1+q+\cdots+q^{k-2})}{1+q^k}.
\]
Thus the coefficient sequence $(g_{k,\ell})_{\ell\ge 0}$ is periodic and bounded.
Since $n_2>0$ by Lemma \ref{n2positive}, there is a constant
$C>0$ such that $a_\ell\ge 4n_2\ell-C$ for all $\ell\ge 0$.

We now distinguish the parity of $n$.
First suppose that $n$ is odd.  Then $1+(-q)^n=1-q^n$, so
\[
\Mag(\A,-q)=\frac{H(q)}{1-q^n}=H(q)\sum_{j=0}^\infty q^{nj}.
\]
Therefore, with $N=\lfloor \ell/n\rfloor$,
\[
e_\ell=\sum_{j=0}^{N}a_{\ell-nj}
     \ge \sum_{j=0}^{N}\bigl(4n_2(\ell-nj)-C\bigr).
\]
The right-hand side is a polynomial of degree $1$ in $\ell$ with leading coefficient $4n_2(N+1)>0$, so it tends to $+\infty$ as $\ell\to\infty$.
Hence $e_\ell>0$ for all
sufficiently large $\ell$ in this case.

Next suppose that $n$ is even.  Then $1+(-q)^n=1+q^n$ and
\[
\frac{1}{1+q^n}=\sum_{j=0}^\infty (-1)^j q^{nj}.
\]
For even $k$, let $b_{k,\ell}$ be the coefficient of $q^\ell$ in the series expansion of $\frac{G_k(q)}{1+q^n}$.
Then we have
\[
b_{k,\ell}=\sum_{j=0}^{\lfloor \ell/n\rfloor}(-1)^jg_{k,\ell-nj}.
\]
Since $(g_{k,\ell})_\ell$ is non-negative and non-decreasing, pairing adjacent
terms gives $b_{k,\ell}\ge 0$ for every $\ell$.
For $k=2$, using $g_{2,m}=m$, we also get $b_{2,\ell}\to +\infty$.
We claim that the odd $k$ terms, after division by $1+q^n$, are
bounded.  Indeed, for odd $k$ and even $n$,
\[
\frac{G_k(q)}{1+q^n}
 =\frac{q(1+q+\cdots+q^{k-2})}{(1+q^k)(1+q^n)}.
\]
Moreover $\gcd(1+q^k,1+q^n)=1$, because $k$ is odd and $n$ is even.  Hence the
denominator has only simple roots of unity, and the coefficient sequence is
bounded.  The coefficient sequence of $\#\Ch(\A)/(1+q^n)$ is also bounded.  Thus there
is a constant $D>0$ such that the total contribution of $\#\Ch(\A)$ and all odd $k$
terms to $e_\ell$ is at least $-D$ for every $\ell$.
Combining these observations, and using Lemma \ref{n2positive}, we obtain
\[
e_\ell\ge 4n_2 b_{2,\ell}-D.
\]
Since $b_{2,\ell}\to +\infty$, the right-hand side is positive for all
sufficiently large $\ell$.  This proves the theorem.
\end{proof}

\subsection{Arrangements with non-alternating magnitude}
In this section, we use the construction of Burr--Gr\"unbaum--Sloane \cite{BGS} in the orchard problem to construct an example of an arrangement whose magnitude does not have the alternating-sign property.

Recall the formula \eqref{mag_minus_q} for $\Mag(\A,-q)$.
The series expansion of $G_3(q)$ is given by
\[
    G_3(q)=\dfrac{q}{1-q+q^2}=q+q^2-q^4-q^5+q^7+q^8-\cdots.
\]
This suggests that for an arrangement with a large $n_3$, the series expansion of $\Mag(\A,-q)$ might have negative coefficients.
This motivates the construction given below.

\begin{defn}
Let $S\subset \R^2$ be a finite set of points.
For each $k\geq 2$, define $m_k$ to be the number of lines that contain exactly $k$ points from $S$.
Define an arrangement $\A(S)$ in $\R^3$ by
\[
\A(S)=\{(x,y,1)^\perp\mid (x,y)\in S\}.
\]
The arrangement $\A(S)$ is essential unless $S$ is contained in a single line.
It satisfies $n_k=2m_k$ for $k\geq 2$.
\end{defn}

We use the following result to construct an $S$ with large $m_3$.

\begin{thm}[{\cite[Theorem 1]{BGS}}, see also {\cite[Proposition 2.6]{GreenTao}}]\label{thm:large_m3}
    For any $N\geq 3$, there exists a finite set $S\subset \R^2$ with $\#S=N$ such that
    \[
    m_k=\begin{cases}
        N-1-2\cdot \mathbf{1}_{3\mid N},&\text{ if }k=2;\\
        \biggl\lfloor\dfrac{N(N-3)}{6}+1\biggr\rfloor,&\text{ if }k=3;\\
        0,&\text{ if }k\geq 4,
    \end{cases}
    \]
    where $\mathbf{1}_{3\mid N}$ equals $1$ if $3$ divides $N$ and $0$ otherwise.
\end{thm}

\begin{ex}\label{ex:18}
    As an example of Theorem \ref{thm:large_m3}, we construct a set of points $S\subset \mathbb{R}^2$ with $\#S=18$ such that
    \[
    m_k=\begin{cases}
        15,&\text{ if }k=2;\\
        46,&\text{ if }k=3;\\
        0,&\text{ if }k\geq 4.
    \end{cases}
    \]
    Let $X$ be the set of nonsingular points of the real projective algebraic curve
    \[
    Y^2=X^2(X-Z).
    \]
    Then the points of $X$ naturally form a group isomorphic to $\mathbb{R}/\mathbb{Z}$.
    An explicit parametrization is given by
    \[
    \phi(t)=\bigl[\cot^2(\pi t)+1:\cot(\pi t)(\cot^2(\pi t)+1):1\bigr]\quad (t\in \mathbb{R}/\mathbb{Z}).
    \]
    Let $P_j\in \R^2$ be the image of $\phi(j/18)$ under the projective transformation
    $[X:Y:Z]\mapsto [X:Y:X+Y+Z]$.
    Then $P_i,P_j,P_k$ lie on a single line if and only if $i+j+k\equiv 0 \pmod{18}$.
    Define $S\subset \R^2$ by
    \[
    S=\{P_j\mid j=0,1,\dots,17\}.
    \]
    There are exactly $46$ triples $\{i,j,k\}\subset \mathbb{Z}/18\mathbb{Z}$ such that $i+j+k\equiv 0\pmod{18}$, so we have $m_3=46$.
    Also, $m_k=0$ for $k\geq 4$.
    Since $m_2+3m_3=\binom{18}{2}$ by double-counting, we have $m_2=15$.
\end{ex}

\begin{ex}
    We give an explicit construction of an arrangement $\A$ such that the series expansion of $\Mag(\A)$ does not have the alternating-sign property.
    Let $S\subset \R^2$ be a finite set as in Example \ref{ex:18}.
    Then the associated arrangement $\A(S)$ satisfies
    \[
    \#\A(S)=18,\quad n_k=
    \begin{cases}
        30,&\text{ if }k=2;\\
        92,&\text{ if }k=3;\\
        0,&\text{ if }k\geq 4,
    \end{cases}\quad\text{and}\quad \#\Ch(\A)=216.
    \]
    Figure \ref{fig:ex18} shows the projectivization of $\A(S)$.
    By Theorem \ref{thm:rank3}, the magnitude of $\A(S)$ is given by
    \begin{align*}
        \Mag(\A)&{}=\dfrac{1}{1+q^{18}}\biggl(216-\dfrac{120q}{(1+q)^2}-\dfrac{552q}{1+q+q^2}\biggr)\\
        &{}=216-672q+792q^2-360q^3-72q^4-48q^5+720q^6-1392q^7+1512q^8-\cdots.
    \end{align*}
    In particular, the coefficient of $q^4$ is $-72<0$.
\end{ex}

\begin{figure}
    \centering
    \includegraphics[width=0.35\linewidth]{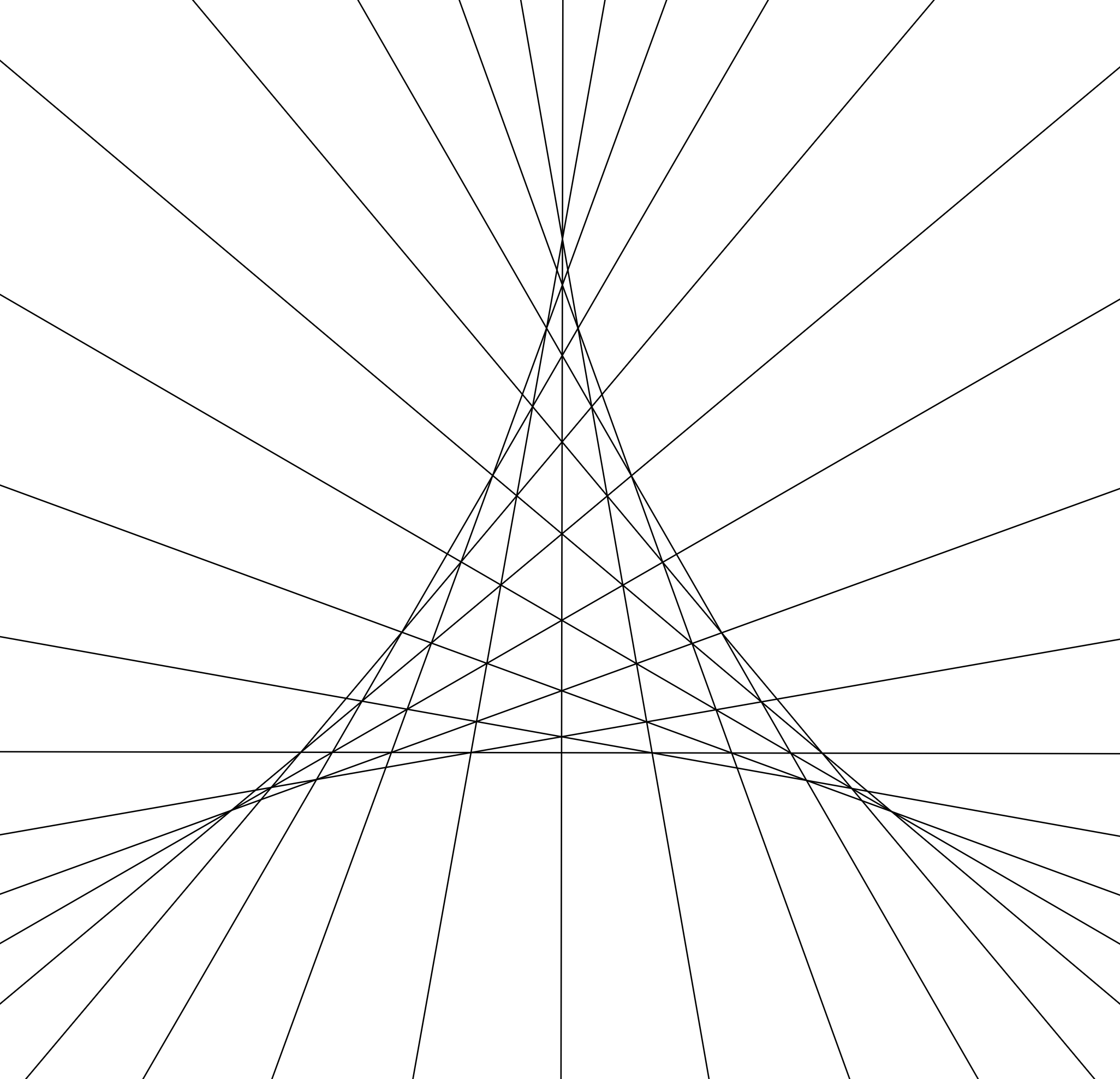}
    \caption{An arrangement whose magnitude does not have the alternating-sign property}
    \label{fig:ex18}
\end{figure}

\section{Magnitude homology of COM-graphs}
Throughout this section, we fix a finite set $E$.
Let $G\subset Q_E$ be a COM-graph.
We define the magnitude homology $MH_{k,\ell}(G)$ of $G$ as the magnitude homology of its underlying graph.

\subsection{Magnitude homology with support}
We first define the following variant of magnitude homology.

\begin{defn}[Magnitude homology with support]\label{def:MH_support_COM}
    Let $G\subset Q_E$ be a COM-graph.
    For a proper chain $\vec{x}=(x_0,\dots,x_k)$ of $G$, we define its support $S(\vec{x})$ to be $S(\{x_0,\dots,x_k\})$.
    For a subset $E'\subset E$, define $MC_{k,\ell}(G;E')$ to be the subgroup of $MC_{k,\ell}(G)$ generated by chains $\vec{x}$ with $S(\vec{x})=E'$.
    The following lemma shows that $MC_{*,\ell}(G;E')$ is a subcomplex of $MC_{*,\ell}(G)$.
    The $E'$-supported magnitude homology $MH_{*,\ell}(G;E')$ is defined to be the homology of the chain complex $MC_{*,\ell}(G;E')$.
\end{defn}

\begin{lem}\label{lem:full_support_boundary}
    Let $G\subset Q_E$ be a COM-graph.
    Let $\vec{x}=(x_0,\dots,x_k)$ be a proper chain of $G$ with $S(\vec{x})=E'$.
    If $(x_{i-1},x_i,x_{i+1})$ is geodesic, then $\partial_i\vec{x}=(x_0,\dots,\widehat{x_i},\dots,x_k)$ has the same support $S(\partial_i\vec{x})=E'$.
\end{lem}

\begin{proof}
Let $a,b,c\in Q_E$.
If $(a,b,c)$ is geodesic, then for every coordinate $e\in E$ with $a_e=c_e$, we have $b_e=a_e=c_e$. In particular, $\{a_e,b_e,c_e\}=\{a_e,c_e\}$ holds for every $e\in E$.
Applying this to $(a,b,c)=(x_{i-1},x_i,x_{i+1})$ proves the assertion.
\end{proof}

By definition, the magnitude homology $MH_{*,\ell}(G)$ decomposes as
\[
MH_{*,\ell}(G)\cong \bigoplus_{E'\subset E}MH_{*,\ell}(G;E').
\]
The main result of this section is the following. The remainder of this section is devoted to its proof.

\begin{thm}[Vanishing theorem]\label{thm:COM_vanish}
    Let $G\subset Q_E$ be a realizable COM-graph.
    Suppose that $G$ has no antipodal subgraph $H\subset G$ with $\supp(H)=E$.
    Then we have
    \[
    MH_{k,\ell}(G;E)=0\quad (\forall k,\ell\geq 0).
    \]
\end{thm}

\subsection{First reduction}
Let $G\subset Q_E$ be a realizable COM-graph.
Suppose that $G$ has no antipodal subgraph $H\subset G$ with $\supp(H)=E$.
We want to show that $MH_{k,\ell}(G;E)=0\ (\forall k,\ell\geq 0)$.

By Theorem \ref{thm:corner-peeling}, $G$ admits a corner peeling $G=C_1\sqcup\cdots\sqcup C_m$.
Let $G_i=G\setminus (C_1\sqcup\cdots\sqcup C_{i-1})$.
Then we have a sequence of inclusions of chain complexes
\[
0=MC_{*,\ell}(G_{m+1};E)\hookrightarrow
MC_{*,\ell}(G_m;E)\hookrightarrow
\cdots\hookrightarrow
MC_{*,\ell}(G_1;E)=MC_{*,\ell}(G;E).
\]
It suffices to show that each step $MC_{*,\ell}(G_{i+1};E)\hookrightarrow MC_{*,\ell}(G_i;E)$ is a quasi-isomorphism.
Since each $G_i$ also satisfies the assumption of Theorem \ref{thm:COM_vanish}, we may reduce to the following claim:

\begin{thm}\label{thm:corner-excision}
    Let $G\subset Q_E$ be a COM-graph and suppose that $G$ has no antipodal subgraph $H\subset G$ with $\supp(H)=E$.
    Let $C$ be a corner of $G$ and set $G'=G\setminus C$.
    Then the inclusion $MC_{*,\ell}(G';E)\hookrightarrow MC_{*,\ell}(G;E)$ is a quasi-isomorphism.
\end{thm}

Let $G\subset Q_E$ and $C\subset G$ be as in the statement of Theorem \ref{thm:corner-excision}.
Then we have the following:
\begin{enumerate}
        \item There is a unique maximal antipodal subgraph $H\subset G$ which contains $C$.
        \item Since $G$ is a COM-graph, $H$ admits a gate projection $\pi\colon G\to H$.
        \item No vertex of $C$ is adjacent to a vertex of $G\setminus H$.
\end{enumerate}

\begin{lem}\label{lem:gate-avoids-corner}
If $y\in G\setminus H$, then $\pi(y)\in H\setminus C$.
\end{lem}

\begin{proof}
Suppose, for a contradiction, that $\pi(y)=c\in C$.  Let
\[
  y=v_0,v_1,\ldots,v_m=c
\]
be a shortest path from $y$ to $c$.
Since $y\notin H$ and $c\in H$, we have $m>0$.  The vertex $v_{m-1}$ is adjacent to $c$.
Since no vertex of $C$ is adjacent to a vertex of $G\setminus H$, we have $v_{m-1}\in H$.
This contradicts the fact that $c=\pi(y)$ is a closest vertex of $H$ to $y$.  Therefore $\pi(y)\notin C$.
\end{proof}

\subsection{Discrete Morse matching}
Let $G,C,H$ be as above.
Let $G'=G\setminus C$. Define
\[
MC_{*,\ell}(G/G';E):=MC_{*,\ell}(G;E)/MC_{*,\ell}(G';E).
\]
We prove the acyclicity of $MC_{*,\ell}(G/G';E)$ using the discrete Morse theorem (Theorem \ref{thm:complete-morse-matching}).
The group $MC_{k,\ell}(G/G';E)$ is freely generated by the full-support chains
\[
  \vec{x}=(x_0,\ldots,x_k),\quad \supp(\vec{x})=E,
\]
which contain at least one vertex of $C$.
Boundary terms that no longer contain a vertex of $C$ are zero in the quotient.  By Lemma \ref{lem:full_support_boundary}, every nonzero boundary term of a full-support chain is still full-support.
Let $\B_{k,\ell}$ denote the above basis of $MC_{k,\ell}(G/G';E)$, and let $\B=\bigcup_k\B_{k,\ell}$.

\begin{defn}
    Let $\vec{x}=(x_0,\ldots,x_k)\in\B_{k,\ell}$.
    Let
    \[
    j=j(\vec{x})=\min\{i\mid x_i\in C\}
    \]
    be the first occurrence of a vertex of $C$ in $\vec{x}$.
\end{defn}

\begin{defn}[Discrete Morse matching]
    Let $\vec{x}=(x_0,\ldots,x_k)\in\B_{k,\ell}$ and $j=j(\vec{x})$.
    Since $\supp(H)\subsetneq E=\supp(\vec{x})$, there exists some $i$ with $x_i\notin H$.
    \begin{enumerate}
        \item If there exists $i<j$ with $x_i\not\in H$, then we say that $\vec{x}$ is a \emph{left chain}.
        In this case, we let
        \[
        r=r(\vec{x})=\max\{i<j\mid x_i\notin H\}.
        \]
        Then $x_r\notin H$ and $x_{r+1},x_{r+2},\ldots,x_j\in H$.
        If $\pi(x_r)\neq x_{r+1}$, then we call $\vec{x}$ a \emph{left lower chain} and pair it with
        \[
        (x_0,\dots,x_r,\pi(x_r),x_{r+1},\dots,x_k) \in \B_{k+1,\ell}.
        \]
        Since $\pi$ is the gate projection to $H$, we have $d(x_r,x_{r+1})=d(x_r,\pi(x_r))+d(\pi(x_r),x_{r+1})$ and hence the length of the chain does not change.
        If $\pi(x_r)=x_{r+1}$, we call $\vec{x}$ a \emph{left upper chain} and pair it with
        \[
        (x_0,\dots,x_r,x_{r+2},\dots,x_k)\in \B_{k-1,\ell}.
        \]
        This also preserves the length of the chain.
        Moreover, the resulting chain still contains a vertex of $C$ since $\pi(x_r)\notin C$ by Lemma \ref{lem:gate-avoids-corner}.
        \item If there is no $i<j$ with $x_i\not\in H$, then we say that $\vec{x}$ is a \emph{right chain}.
        In this case, there exists some $i>j$ with $x_i\notin H$.
        Let
        \[
        r=r(\vec{x})=\min\{i>j\mid x_i\notin H\}.
        \]
        Then $x_j,x_{j+1},\ldots,x_{r-1}\in H$ and $x_r\notin H$.
        If $\pi(x_r)\neq x_{r-1}$, then we call $\vec{x}$ a \emph{right lower chain} and pair it with
        \[
        (x_0,\ldots,x_{r-1},\pi(x_r),x_r,\ldots,x_k) \in \B_{k+1,\ell}.
        \]
        Since $\pi$ is the gate projection to $H$, we have $d(x_{r-1},x_r)=d(x_{r-1},\pi(x_r))+d(\pi(x_r),x_r)$ and hence the length of the chain does not change.
        If $\pi(x_r)=x_{r-1}$, then we call $\vec{x}$ a \emph{right upper chain} and pair it with
        \[
        (x_0,\ldots,x_{r-2},x_r,\ldots,x_k)\in \B_{k-1,\ell}.
        \]
        This also preserves the length of the chain.
        Moreover, the resulting chain still contains a vertex of $C$ since $\pi(x_r)\notin C$ by Lemma \ref{lem:gate-avoids-corner}.
    \end{enumerate}
    This defines a complete matching $M$ on $\B$ such that the weight of each matched arrow is $\pm 1$.
\end{defn}

\subsection{Proof of acyclicity}
It remains to check the acyclicity of the matching $M$ constructed above.
Let us show that there is no directed cycle in the directed graph $\Gamma_M$.
It suffices to show that there is no directed cycle of the form
\begin{align}
  \vec{x}^{(0)}\to\vec{y}^{(0)}\to \vec{x}^{(1)}\to \vec{y}^{(1)}\to\cdots\to \vec{x}^{(n)}\to\vec{y}^{(n)}\to\vec{x}^{(n+1)}=\vec{x}^{(0)},\label{eq:directed_cycle}
\end{align}
where each $\vec{x}^{(i)}$ is a lower chain, each $\vec{y}^{(i)}$ is the matched upper chain, and $\vec{x}^{(i+1)}$ is a boundary face of $\vec{y}^{(i)}$ different from $\vec{x}^{(i)}$.

Supporse, for a contradiction, that a directed cycle of the form \eqref{eq:directed_cycle} exists.
For a lower chain $\vec{x}$, define $N_C(\vec{x})=\#\{i\mid x_i\in C\}$.
Then one checks that $N_C(\vec{x}^{(i)})\geq N_C(\vec{x}^{(i+1)})$ holds.
Therefore we must have $N_C(\vec{x}^{(0)})=\cdots=N_C(\vec{x}^{(n)})$.
Also define an integer $\omega(\vec{x})$ as follows:
\begin{enumerate}
    \item If $\vec{x}$ is a left lower chain, then $\omega(\vec{x})=r(\vec{x})$.
    \item If $\vec{x}$ is a right lower chain, then $\omega(\vec{x})=-r(\vec{x})$.
\end{enumerate}

\begin{lem}
    For a directed cycle of the form \eqref{eq:directed_cycle}, we have $\omega(\vec{x}^{(i+1)})<\omega(\vec{x}^{(i)})$ for $i=0,1,\dots,n$.
\end{lem}

\begin{proof}
First suppose that $\vec{x}^{(i)}$ is a left lower chain.
Write
\[
  \vec{x}^{(i)}=(\dots,x_r,x_{r+1},\dots,x_j,\dots),
  \quad x_r\notin H,\quad x_{r+1},\ldots,x_j\in H.
\]
The matched upper chain is
\[
  \vec{y}^{(i)}=(\dots,x_r,\pi(x_r),x_{r+1},\dots,x_j,\dots).
\]
The chain $\vec{x}^{(i+1)}$ is obtained by deleting a vertex $v$ from $\vec{y}^{(i)}$.
Since $\vec{x}^{(i)}\neq\vec{x}^{(i+1)}$, we have $v\neq \pi(x_r)$.
Since $N_C(\vec{x}^{(i)})=N_C(\vec{x}^{(i+1)})$, we have $v\neq x_j$.
If $v\neq x_r$, then $\vec{x}^{(i+1)}$ remains an upper chain.
Therefore $v$ must be $x_r$.
After deleting $x_r$, either there is still a vertex outside $H$ before $x_j$, in which case $r(\vec{x}^{(i+1)})<r(\vec{x}^{(i)})$, or there is no exterior vertex, in which case $\vec{x}^{(i+1)}$ is a right chain and $\omega(\vec{x}^{(i+1)})<0\leq \omega(\vec{x}^{(i)})$.  In both subcases, we have $\omega(\vec{x}^{(i+1)})<\omega(\vec{x}^{(i)}).$

The right case is almost symmetric.
Write
\[
  \vec{x}^{(i)}=(\dots,x_j,\dots,x_{r-1},x_r,\dots),
  \quad x_r\notin H,\quad x_j,\dots,x_{r-1}\in H.
\]
The matched upper chain is
\[
  \vec{y}^{(i)}=(\dots,x_j,\dots,x_{r-1},\pi(x_r),x_r,\dots).
\]
The chain $\vec{x}^{(i+1)}$ is obtained by deleting a vertex $v$ from $\vec{y}^{(i)}$.
Since $\vec{x}^{(i)}\neq\vec{x}^{(i+1)}$, we have $v\neq \pi(x_r)$.
Since $N_C(\vec{x}^{(i)})=N_C(\vec{x}^{(i+1)})$, we have $v\neq x_j$.
If $v\neq x_r$, then $\vec{x}^{(i+1)}$ remains an upper chain.
Therefore $v$ must be $x_r$.
After deleting $x_r$, there is still a vertex outside $H$ after $x_j$.
The selected index $r(\vec{x}^{(i+1)})$ satisfies $r(\vec{x}^{(i+1)})>r(\vec{x}^{(i)})$, and therefore
\[
  \omega(\vec{x}^{(i+1)})=-r(\vec{x}^{(i+1)})<-r(\vec{x}^{(i)})=\omega(\vec{x}^{(i)}).
\]
This completes the proof.
\end{proof}

This contradiction shows that there is no directed cycle in $\Gamma_M$ and thus the matching $M$ is acyclic.
Therefore, by Theorem \ref{thm:complete-morse-matching}, $MC_{*,\ell}(G/G';E)$ is acyclic and thus Theorem \ref{thm:COM_vanish} is proved.

\section{Magnitude homology of arrangements}

Let $\A$ be an arrangement.
We continue to denote invariants of $\T(\A)$ simply by the corresponding invariants of $\A$, such as $MH_{k,\ell}(\A):=MH_{k,\ell}(\T(\A))$, $\beta_{k,\ell}(\A):=\rank MH_{k,\ell}(\A)$, and $\chi_{\ell}(\A)=\sum_k(-1)^k\beta_{k,\ell}(\A)$. Also recall the notation $c^X=\#\Ch(\A^X)$ and $c_X=\#\Ch(\A_X)$ for $X\in L(\A)$.

\subsection{Small lengths and diagonality}
We begin our study of $MH_{k,\ell}(\A)$ by investigating lengths $\ell\leq 2$.

\begin{prop}\label{smalllength}
    Let $\A$ be an arrangement. Then the only nontrivial $MH_{k,\ell}(\A)$ for $\ell\leq 2$ are as follows.
    \begin{enumerate}
        \item $MH_{0,0}(\A)$ is free abelian of rank $\beta_{0,0}(\A)=\#\Ch(\A)$.
        \item $MH_{1,1}(\A)$ is free abelian of rank $\beta_{1,1}(\A)=2\sum_{H\in \A}c^H$.
        \item $MH_{2,2}(\A)$ is free abelian of rank $\beta_{2,2}(\A)=\beta_{1,1}(\A)+4\sum_{X\in L_2(\A),\#\A_X=2}c^X$.
    \end{enumerate}
\end{prop}

\begin{proof}
    (1) and (2) follow from (1) and (2) of Proposition \ref{lowdeg}, together with the observation that an edge in $\T(\A)$ corresponds to a chamber of the restriction $\A^H$ for some $H\in\A$. 
    
    For (3), since $m_{\T(\A)}\geq 3$, we may apply Theorem \ref{gsmh} for $\ell=2<m_{\T(\A)}$:
    \[
    MH_{2,2}(\A)\cong \bigoplus_{\Vec{a}\in P_{1,2}(\A)\cup P_{2,2}(\A)}MH^{\gs}_{2,\Vec{a}}(\A).
    \]
    By Theorem \ref{fmh}, the framed magnitude homology of $\A$ can be computed from interval posets. For $\Vec{a}=(a_0,a_1)\in P_{1,2}(\A)$, where $a_0$ and $a_1$ are chambers at distance $2$, we may suppose $S(a_0,a_1)=\{H_1,H_2\}$ and let $X=H_1\cap H_2\in L_2(\A)$. If $\#\A_X=2$, then the interval poset $(a_0,a_1)_{\T(\A)}$ consists of two incomparable chambers between $a_0$ and $a_1$. If $\#\A_X>2$, then $(a_0,a_1)_{\T(\A)}$ is a singleton. Since $MH^{\gs}_{2,(a_0,a_1)}(\A)\cong \widetilde{H}_0(C(a_0,a_1)_{\T(\A)})$, taking reduced homology shows that only the first type of $(a_0,a_1)$, with $\#\A_X=2$, contributes a copy of $\Z$ to the magnitude homology. Hence the total contribution to $MH_{2,2}(\A)$ from $P_{1,2}(\A)$ is $4\sum_{X\in L_2(\A),\#\A_X=2}c^X$ copies of $\Z$.
    
    Similarly, for $\Vec{a}=(a_0,a_1,a_2)\in P_{2,2}(\A)$, if $a_0\neq a_2$, then $\Vec{a}$ is geodesic and hence $\varphi(\Vec{a})\neq \Vec{a}$, making the $\Vec{a}$-framed magnitude homology trivial. Thus we only need to consider $\Vec{a}=(a_0,a_1,a_0)\in P_{2,2}(\A)$, where $a_0$ and $a_1$ are adjacent. The interval posets $(a_0,a_1)_{\T(\A)}$ and $(a_1,a_0)_{\T(\A)}$ are empty. Then $MH^{\gs}_{2,(a_0,a_1,a_0)}(\A)\cong \widetilde{H}_{-2}(C(a_0,a_1)_{\T(\A)}\otimes C(a_1,a_0)_{\T(\A)})\cong \Z$. Each such $(a_0,a_1,a_0)$ corresponds to a directed edge of $\T(\A)$. Hence the total contribution to $MH_{2,2}(\A)$ from $P_{2,2}(\A)$ is $\beta_{1,1}(\A)$ copies of $\Z$. This proves (3).

    To finish the proof, we need to show that all other $MH_{k,\ell}(\A)$ for $\ell\leq 2$ are trivial. By Proposition \ref{lowdeg} (3), we only need to prove $MH_{1,2}(\A)=0$. We use Theorem \ref{gsmh} again:
    \[
    MH_{1,2}(\A)\cong \bigoplus_{\Vec{a}\in P_{1,2}(\A)}MH^{\gs}_{1,\Vec{a}}(\A).
    \]
    Consider $\Vec{a}=(a_0,a_1)\in P_{1,2}(\A)$. The interval poset $(a_0,a_1)_{\T(\A)}$ cannot be empty. Hence $MH^{\gs}_{1,(a_0,a_1)}(\A)\cong \widetilde{H}_{-1}(C(a_0,a_1)_{\T(\A)})=0$.
\end{proof}

Next we determine the geodesic magnitude homology of arrangements.

\begin{thm}
    Let $\A$ be an arrangement. Then its geodesic magnitude homology is given as follows:
    \[
    MH_{k,\ell}^{\geod}(\A)\cong \bigoplus_{\substack{F\in\FF(\A)\\ \rank \A_F=k,\ \#\A_F=\ell}}\Z^{\#\Ch(\A_F)}
    \cong
    \bigoplus_{\substack{X\in L(\A)\\\rank \A_X=k,\ \#\A_X=\ell}}\Z^{c^Xc_X}.
    \]
    In particular, $MH_{k,\ell}^{\geod}(\A)$ is determined by $L(\A)$.
\end{thm}

\begin{proof}
    By Corollary \ref{geodMHdecomp}, we have
    \[
    MH_{k,\ell}^{\geod}(\A)\cong \bigoplus_{\substack{a,b\in\Ch(\A)\\ d(a,b)=\ell}}\widetilde{H}_{k-2}(C_*(a,b)_{\T(\A)}).
    \]
    By Proposition \ref{intervelhomotopy}, the interval $(a,b)_{\T(\A)}$ is homotopy equivalent to $S^{\rank \A_F-2}$ if $[a,b]_{\T(\A)}=\Ch(\A)_F$ for some $F\in\FF(\A)$, and is contractible otherwise. 
    Therefore each $F\in\FF(\A)$ such that $k=\rank \A_F$ and $\ell=\#\A_F$ contributes a copy of $\Z^{\#\Ch(\A)_F}$ to $MH_{k,\ell}^{\geod}(\A)$, and there are no other summands.
\end{proof}

\begin{cor}\label{geodbetti}
    Let $\A$ be an arrangement. Then its geodesic magnitude Betti numbers are given by
    \[
    \beta_{k,\ell}^{\geod}(\A)=\sum_{\substack{X\in L(\A)\\ \rank\A_X=k,\ \#\A_X=\ell}}c^Xc_X.
    \]
\end{cor}

As an application, we give a characterization of diagonal tope graphs.

\begin{cor}\label{booleaniffdiagonal}
     Let $\A$ be an essential arrangement. Then $\T(\A)$ is diagonal if and only if $\A$ is Boolean.
\end{cor}
\begin{proof}
    We know that the Boolean arrangement $Bl_d$ has tope graph $\T(Bl_d)=Q_d$, which is diagonal (see Example \ref{hypercube}). Suppose $\A$ is not Boolean. Then $n=\#\A>r=\rank \A$. We have
    \[
    \beta_{r,n}(\A)\geq \beta_{r,n}^{\geod}(\A)=\#\Ch(\A),
    \]
    where the equality follows from Corollary \ref{geodbetti}. This means that $\A$ is not diagonal.
\end{proof}

\subsection{Magnitude homology with support}

For a proper chain $\Vec{x}=(x_0,\ldots,x_k)\in P_{k,\ell}(\A)$, we define its support $S(\Vec{x})$ to be $S(x_0,x_1)\cup\cdots\cup S(x_{k-1},x_k)$.
The following definition is a special case of Definition \ref{def:MH_support_COM}.

\begin{defn}[Magnitude homology with support]\label{relativechain}
    For an arrangement $\A$ and a subarrangement $\B\subseteq\A$, we define $MC_{k,\ell}(\A;\B)$ as the subgroup of $MC_{k,\ell}(\A)$ generated by $P_{k,\ell}(\A;\B)=\{\Vec{x}\in P_{k,\ell}(\A)\mid S(\Vec{x})=\B\}$. This defines a subcomplex $MC_{*,\ell}(\A;\B)$ of $MC_{*,\ell}(\A)$, called the \emph{$\B$-supported magnitude complex} of $\A$. Its homology $MH_{*,\ell}(\A;\B)$ is called the \emph{$\B$-supported magnitude homology} of $\A$.
\end{defn}

The following proposition is immediate from the definition.

\begin{prop}\label{wallcrossing}
    For an arrangement $\A$, the magnitude homology $MH_{*,\ell}(\A)$ decomposes as
    \[
        MH_{*,\ell}(\A)=\bigoplus_{\B\subseteq \A} MH_{*,\ell}(\A;\B).
    \]
    Furthermore, $MH_{*,\ell}(\A;\B)=0$ if $\ell<\#\B$.
\end{prop}

Using the notion of magnitude homology with support, we define interior magnitude homology. We later show that this is a categorification of the interior magnitude defined in Definition \ref{def:intmag}.

\begin{defn}[Interior magnitude homology]\label{interiorMH}
    Let $\A$ be an arrangement.
    We define
    \[
    MC_{*,\ell}^{\circ}(\A):=MC_{*,\ell}(\A;\A),
    \]
    and call it the length $\ell$ \emph{interior magnitude complex} of $\A$.
    We call $MH_{*,\ell}^{\circ}(\A):=H_*(MC_{*,\ell}^{\circ}(\A))$ the length $\ell$ \emph{interior magnitude homology} of $\A$.
    Note that $MH^\circ_{*,\ell}(\A)=0$ for $\ell<\#\A$ by definition.
    We also define the \emph{interior magnitude Betti number} $\beta_{k,\ell}^{\circ}(\A):=\rank MH_{k,\ell}^{\circ}(\A)$ and the length $\ell$ \emph{interior magnitude Euler characteristic} $\chi_{\ell}^{\circ}(\A):=\sum_{k}(-1)^k\beta_{k,\ell}^{\circ}(\A)$.
\end{defn}

The following theorem is a homological version of the face decomposition formula given in Theorem \ref{facedecomp}.

\begin{thm}[Homological face decomposition]\label{facedecompMH}
    For an arrangement $\A$, we have the following face decomposition of its magnitude homology:
    \[
    MH_{k,\ell}(\A)\cong \bigoplus_{F\in \mathcal{F}(\A)}MH^\circ_{k,\ell}(\A_F)\quad (\forall k,\ell\geq 0).
    \]
\end{thm}

To prove this theorem, we introduce some auxiliary definitions.

\begin{defn}
Let $\A$ be an arrangement and let $\B\subseteq \A$ be a subarrangement.
Then any $\B$-supported chain is entirely contained in a single chamber of $\A\setminus \B$.
For $D\in \Ch(\A\setminus\B)$, we write $MC_{*,\ell}(\A;\B)_D$ for the subcomplex of $MC_{*,\ell}(\A;\B)$ spanned by chains contained in $D$.
We write $MH_{*,\ell}(\A;\B)_D$ for its homology.
\end{defn}

By definition, $MH_{*,\ell}(\A;\B)$ decomposes as
\[
MH_{*,\ell}(\A;\B)\cong \bigoplus_{D\in \Ch(\A\setminus\B)}MH_{*,\ell}(\A;\B)_D.
\]
The following result is a direct consequence of Theorem \ref{thm:COM_vanish} for COM-graphs.

\begin{prop}\label{prop:vanishing_arr}
    Let $\A$ be an arrangement and $\B\subset \A$ be a subarrangement with $\bigcap \B=X$.
    Let $D$ be a chamber of $\A\setminus\B$.
    If $X\cap D=\emptyset$, then $MH_{*,\ell}(\A;\B)_D=0$.
\end{prop}

\begin{proof}
    Let $G$ be the induced subgraph of $\mathcal{T}(\A)$ spanned by chambers contained in $D$.
    Then $G\subset \{+,-\}^\B$ is a realizable COM-graph, and its $\B$-supported magnitude homology $MH_{*,\ell}(G;\B)$ can be identified with $MH_{*,\ell}(\A;\B)_D$.
    Antipodal subgraphs of $G$ correspond bijectively to faces of $\B$ that have nonempty intersection with $D$.
    In particular, the assumption $X\cap D=\emptyset$ shows that there is no antipodal subgraph $H\subset G$ with $\supp(H)=\B$.
    Therefore Theorem \ref{thm:COM_vanish} implies the desired vanishing $MH_{*,\ell}(G;\B)=0$.
\end{proof}

\begin{prop}\label{prop:face_localization}
    Let $\A$ be an arrangement and $\B\subset \A$ be a subarrangement with $\bigcap \B=X$.
    Let $D$ be a chamber of $\A\setminus\B$.
    If $X\cap D=F\neq\emptyset$, then $MH_{*,\ell}(\A;\B)_D\cong MH_{*,\ell}(\A_F;\B)$.
\end{prop}

\begin{proof}
    If $X\cap D=F\neq \emptyset$, then a chamber $C\in \Ch(\A)$ is contained in $D$ if and only if $F\leq C$.
    Therefore $MC_{*,\ell}(\A;\B)_D$ is generated by the chains contained in the subgraph $\mathcal{T}(\A)_F$ of $\mathcal{T}(\A)$.
    Since $\mathcal{T}(\A)_F$ is a gated subgraph of $\mathcal{T}(\A)$ isomorphic to $\mathcal{T}(\A_F)$, there is an isomorphism of chain complexes $MC_{*,\ell}(\A;\B)_D\cong MC_{*,\ell}(\A_F;\B)$.
    This induces an isomorphism $MH_{k,\ell}(\A;\B)_D\cong MH_{k,\ell}(\A_F;\B)$.
\end{proof}

\begin{cor}\label{cor:MH_with_support}
    Let $\A$ be an arrangement and $\B\subset \A$ be a subarrangement with $\bigcap\B=X$.
    \begin{enumerate}
        \item If $\B=\A_X$, then $MH_{k,\ell}(\A;\B)$ decomposes as
        \[
        MH_{k,\ell}(\A;\B)\cong \bigoplus_{\substack{F\in \mathcal{F}(\A)\\s(F)=X}}MH^\circ_{k,\ell}(\A_F).
        \]
        \item If $\B\neq \A_X$, then $MH_{k,\ell}(\A;\B)=0$.
    \end{enumerate}
\end{cor}

\begin{proof}
    (1) Suppose that $\B=\A_X$.
        By Propositions \ref{prop:vanishing_arr} and \ref{prop:face_localization}, $MH_{k,\ell}(\A;\B)$ decomposes as
        \[
        MH_{k,\ell}(\A;\B)\cong \bigoplus_{\substack{D\in \Ch(\A\setminus\A_X)\\X\cap D=F\neq \emptyset}}MH_{k,\ell}(\A_F;\B).
        \]
        A face $F\in \mathcal{F}(\A)$ can be written as $X\cap D$ if and only if $s(F)=X$.
        Moreover, when $s(F)=X$, there is a unique $D\in \Ch(\A\setminus\A_X)$ such that $X\cap D=F$, and we have $\A_F=\A_X=\B$.
        This proves the claim.
    
    (2) If $\B\neq \A_X$, then there is a hyperplane $H\in \A\setminus \B$ which contains $X$.
        This means that all chambers of $\A\setminus\B$ are disjoint from $X$.
        By Proposition \ref{prop:vanishing_arr}, we conclude that $MH_{k,\ell}(\A;\B)=0$.
\end{proof}

Now we are ready to prove Theorem \ref{facedecompMH}.

\begin{proof}[Proof of Theorem \ref{facedecompMH}]
    By Proposition \ref{wallcrossing} and Corollary \ref{cor:MH_with_support}, the magnitude homology of $\A$ decomposes as
    \[
    MH_{*,\ell}(\A)\cong\bigoplus_{X\in L(\A)}MH_{*,\ell}(\A;\A_X)\cong\bigoplus_{X\in L(\A)}\bigoplus_{\substack{F\in \mathcal{F}(\A)\\s(F)=X}}MH^\circ_{*,\ell}(\A_F).
    \]
    By rewriting the double direct sum on the right-hand side as a single direct sum over $F\in \mathcal{F}(\A)$, we obtain Theorem \ref{facedecompMH}.
\end{proof}

\begin{cor}\label{reformulation}
Let $\A$ be an arrangement with $\rank \A=r$ and $\#\A=n$.
Then we have
\[
\Mag^\circ(\A)=\sum_{\ell=0}^{\infty}\chi_{\ell}^{\circ}(\A)q^{\ell}.
\]
Moreover, we have $\chi_{\ell}^{\circ}(\A)=(-1)^r\chi_{\ell-n}(\A)$ for $\ell\geq n$, and $\chi_{\ell}^{\circ}(\A)=0$ for $\ell<n$.
\end{cor}

\begin{proof}
    Define $\Delta(\A)$ to be the difference of the two expressions:
    \[
    \Delta(\A)=\Mag^\circ(\A)-\sum_{\ell=0}^\infty\chi_\ell^\circ(\A)q^\ell.
    \]
    We prove $\Delta(\A)=0$ by induction on the rank of $\A$.
    By Theorem \ref{facedecomp} and Theorem \ref{facedecompMH}, we have
    \[
    \Mag(\A)=\sum_{F\in \FF(\A)}\Mag^\circ(\A_F), \quad \chi_\ell(\A)=\sum_{F\in \FF(\A)}\chi^\circ_{\ell}(\A_F).
    \]
    Since $\Mag(\A)$ can be written as $\sum_\ell \chi_\ell(\A)q^\ell$, we obtain
    $\sum_{F\in \FF(\A)}\Delta(\A_F)=0$.
    Applying M\"obius inversion to $\FF(\A)$ yields $\Delta(\A)=0$.
    The final assertion follows from the identity $\Mag^\circ(\A)=(-1)^rq^n\Mag(\A)$.
\end{proof}

\begin{rem}
    The identity $\chi_{\ell}^{\circ}(\A)=(-1)^r\chi_{\ell-n}(\A)$ in Corollary \ref{reformulation} is analogous to the Ehrhart--Macdonald reciprocity theorem for lattice polytopes \cite{Ehrhart1967Reciprocity,Macdonald1971} and Stanley's reciprocity theorem for posets \cite{Stanley1970}.
\end{rem}

\begin{rem}\label{magcpxintmagcpx}
    In light of Corollary \ref{reformulation}, it is natural to expect that there is an isomorphism
    \[
    MH^\circ_{k,\ell}(\A)\cong MH_{k-r,\ell-n}(\A)
    \]
    for $k\geq r$ and $\ell\geq n$, and $MH^\circ_{k,\ell}(\A)=0$ otherwise.
\end{rem}

\subsection{Magnitude homology on the diagonal}

In this subsection, we give a combinatorial formula for the diagonal magnitude homology $MH_{\ell,\ell}(\A)$ of an arrangement $\A$.
First, we determine the interior magnitude homology of Boolean arrangements.

\begin{cor}\label{intMHbool}
    For the Boolean arrangement $Bl_d$, the interior magnitude homology $MH_{k,\ell}^{\circ}(Bl_d)$ is concentrated on the diagonal $k=\ell$, and $MH_{\ell,\ell}^{\circ}(Bl_d)$ is free abelian of rank
    \[
    \beta_{\ell,\ell}^{\circ}(Bl_d)=2^d\binom{\ell-1}{d-1}.
    \]
\end{cor}

\begin{proof}
    We know from Corollary \ref{booleaniffdiagonal} and Example \ref{hypercube} that $\T(Bl_d)$ is diagonal and $MH_{\ell,\ell}(Bl_d)$ is free abelian of rank $\beta_{\ell,\ell}(Bl_d)=2^d\binom{\ell+d-1}{d-1}$.
    Since the interior magnitude homology $MH_{k,\ell}^{\circ}(Bl_d)$ is a direct summand of $MH_{k,\ell}(Bl_d)$, it is concentrated on the diagonal $k=\ell$.
    In particular, we have
    \[
    \chi_\ell(Bl_d)=(-1)^\ell\beta_{\ell,\ell}(Bl_d),\quad \chi^\circ_{\ell}(Bl_d)=(-1)^\ell\beta^\circ_{\ell,\ell}(Bl_d).
    \]
    By Corollary \ref{reformulation}, we have $\chi^\circ_{\ell}(Bl_d)=(-1)^d\chi_{\ell-d}(Bl_d)$ and hence
    \[
    \beta^\circ _{\ell,\ell}(Bl_d)=\beta_{\ell-d,\ell-d}(Bl_d)=2^d\binom{\ell-1}{d-1}
    \]
    for $\ell\geq d$.
    Combining this with $\beta^\circ_{\ell,\ell}(Bl_d)=0\ (\ell<d)$, we obtain the desired formula.
\end{proof}

Next we prove the following vanishing result for the diagonal interior magnitude homology of non-Boolean arrangements.
\begin{thm}\label{nonbooldiagintMH}
    Let $\A$ be an arrangement of $\rank \A=r$ and $\#\A=n$. Suppose $\A$ is not Boolean, so that $r<n$. Then the diagonal interior magnitude homology of $\A$ vanishes in all lengths:
    \[
    MH_{\ell,\ell}^{\circ}(\A)=0,\quad \forall \ell\geq 0.
    \]
\end{thm}

\begin{proof}
    Take $c\in MC_{\ell,\ell}^{\circ}(\A)\cap\ker\partial\subseteq MC_{\ell,\ell}(\A)$.
    We will show that $c=0$.
    By the endpoint decomposition \eqref{vertexdecomp}, we may assume that $c$ is a linear combination of chains $\vec{x}=(x_0,\dots,x_\ell)$ with $x_0=s$ and $x_\ell=t$ for fixed $s,t\in \T(\A)$.
    
    Let $\phi:\T(\A)\hookrightarrow Q_{n}=\{+,-\}^n$ be an isometric embedding of $\T(\A)$ into the hypercube $Q_n$, which is not surjective since $r<n$. Since $\phi$ is distance-preserving, it induces injective chain maps $\phi_{\#}:MC_{*,\ell}(\A)\hookrightarrow MC_{*,\ell}(Q_n)$.
    Therefore it suffices to show that $\phi_\#(c)=0$.
    Since $\partial c=0$, $\phi_{\#}(c)\in MC_{\ell,\ell}(Q_n)$ is a cycle. For a fixed $\ell$, we have $MC_{\ell+1,\ell}(Q_n)=0$, so $MH_{\ell,\ell}(Q_n)$ can be regarded as a subgroup of $MC_{\ell,\ell}(Q_n)$.
    We regard $\phi_{\#}(c)$ as an element of $MH_{\ell,\ell}(Q_n)$.

    By the K\"unneth formula (Theorem \ref{Kunneth}), we have an isomorphism
    \[
    \nabla:\bigoplus_{\ell_1+\cdots+\ell_n=\ell}MH_{\ell_1,\ell_1}(Q_1)\otimes\cdots\otimes MH_{\ell_n,\ell_n}(Q_1) \xrightarrow{\cong}MH_{\ell,\ell}(Q_n)
    \]
    which is induced by the Eilenberg--Zilber chain map \cite{Hepworth2017}.
    This restricts to an isomorphism
    \[
    \nabla:\bigoplus_{\ell_1+\cdots+\ell_n=\ell}MH_{\ell_1,\ell_1}(Q_1;s_1,t_1)\otimes\cdots\otimes MH_{\ell_n,\ell_n}(Q_1;s_n,t_n) \xrightarrow{\cong}MH_{\ell,\ell}(Q_n;s,t).
    \]
    Each direct summand on the left-hand side is either $0$ or $\mathbb{Z}$.
    In the latter case, we write $b^\lambda$ for the canonical generator of the summand, where $\lambda=(\ell_1,\ldots,\ell_n)$.
    Then we can write
    \[
    \phi_{\#}(c)=\sum_{\lambda}d_{\lambda}\nabla(b^{\lambda})
    \]
    for some scalars $d_{\lambda}$. We must show that $d_{\lambda}=0$.
    
    Since $c\in MC_{\ell,\ell}^{\circ}(\A)$, it is a linear combination of walks that cross all hyperplanes in $\A$. Its image $\phi_{\#}(c)$ in $MC_{\ell,\ell}(Q_n)$ is a linear combination of walks in $Q_n$ along which every coordinate gets flipped at least once. This means the basis elements $b^{\lambda}$ involved are indexed by $\lambda=(\ell_1,\ldots,\ell_n)$ with all $\ell_i\geq 1$. We write such $b^{\lambda}$ as $b^\lambda=b_1^{\lambda}\otimes\cdots\otimes b_n^{\lambda}$, where $b_i^{\lambda}\in MH_{\ell_i,\ell_i}(Q_1;s_i,t_i)$ is an $\ell_i$-step walk in the $i$-th $Q_1$. Thus $b^{\lambda}$ consists of $\ell$ steps in these $n$ dimensions, with distribution $\lambda$.
    
    By definition of the Eilenberg--Zilber map, $\nabla(b^{\lambda})$ is the signed sum of all possible shuffles of all these $\ell$ steps. It is a linear combination of walks in $Q_n$ from $s$ to $t$.
    Since all $\ell_i\geq 1$, any vertex in $Q_n$ must be visited by at least one walk in $\nabla(b^{\lambda})$. The next key observation is that if $\lambda\neq \mu$, then $\nabla(b^{\lambda})$ and $\nabla(b^{\mu})$ share no common walk because their step distributions are different.

    Now recall that $\phi:\T(\A)\hookrightarrow Q_n$ is not surjective.
    Thus there is a vertex $v_0\in Q_n$ that is not a vertex of $\T(\A)$.
    Observe that any $\nabla(b^{\lambda})$ contains a walk passing through $v_0$ (see Figure \ref{fig:ghost_trap}), but $c\in MC_{\ell,\ell}(\A)$ contains no walk passing through $v_0$.
    Therefore we must have $d_{\lambda}=0$ in the sum $\sum_{\lambda}d_{\lambda}\nabla(b^{\lambda})=\phi_{\#}(c)$.
    Hence we conclude $c=0$, and thus $MH_{\ell,\ell}^{\circ}(\A)=0$.
\end{proof}

\begin{figure}[htbp]
    \centering
    % Define the 3D perspective
    \begin{tikzpicture}[x={(1.5cm,0cm)}, y={(0cm,1.5cm)}, z={(-0.6cm,-0.5cm)}]
        
        % 1. Draw the background cube edges (Q_3) in light gray
        % Back face
        \draw[black] (-1,-1,-1) -- (1,-1,-1) -- (1,1,-1) -- (-1,1,-1) -- cycle;
        % Front face
        \draw[black] (-1,-1,1) -- (1,-1,1) -- (1,1,1) -- (-1,1,1) -- cycle;
        % Connecting edges
        \draw[black] (-1,-1,-1) -- (-1,-1,1);
        \draw[black] (1,-1,-1) -- (1,-1,1);
        \draw[black] (1,1,-1) -- (1,1,1);
        \draw[black] (-1,1,-1) -- (-1,1,1);

        % 2. Draw the A_2 Hexagon strictly on the mixed-sign vertices
        \draw[blue!80!black, ultra thick] 
            (1,-1,-1) -- (1,1,-1) -- (-1,1,-1) -- (-1,1,1) -- (-1,-1,1) -- (1,-1,1) -- cycle;

        % 3. Mark the missing "Ghost" vertices
        \fill[red!90!black] (1,1,1) circle (3pt) node[fill=white, right=0.1, text=red!90!black, font=\bfseries] at (1,0.9,1) {\small $G^+ (+,+,+)$};
        \fill[red!90!black] (-1,-1,-1) circle (3pt) node[fill=white, left=0.1, text=red!90!black, font=\bfseries] at (-1,-0.85,-1) {\small $G^- (-,-,-)$};

        % 4. Mark the Hexagon vertices
        \fill[blue!80!black] (1,-1,-1) circle (2.5pt) node[right, text=black] {\small $(+,-,-)$};
        \fill[blue!80!black] (1,1,-1) circle (2.5pt) node[above, text=black, yshift=2pt] {\small Start: $s=(+,+,-)$};
        \fill[blue!80!black] (-1,1,-1) circle (2.5pt) node[above left, text=black] {\small $(-,+,-)$};
        \fill[blue!80!black] (-1,1,1) circle (2.5pt) node[left, text=black] {\small $(-,+,+)$};
        \fill[blue!80!black] (-1,-1,1) circle (2.5pt) node[left, text=black] {\small $(-,-,+)$}; 
        \fill[blue!80!black] (1,-1,1) circle (2.5pt) node[below right, text=black] {\small End: $t=(+,-,+)$};

        % 5. Draw the fatal Shuffle Path
        % Step 1: Flips Z (Hits the Ghost!)
        \draw[-{Latex[length=4mm]}, ultra thick, orange!90!black, dashed] (1,1,-1) -- (1,1,1) 
            node[fill=white, midway, right=0.2] {\small Step 1 ($Z_1$)};
        % Step 2: Flips X (Moving to valid vertex)
        \draw[-{Latex[length=4mm]}, ultra thick, orange!90!black, dashed, dashed] (1,1,1) to[bend right=20] 
            node[fill=white, above] {\small Step 2 ($X_1$)} (-1,1,1);
        % Step 3: Flips X again (Bouncing back to Ghost)
        \draw[-{Latex[length=4mm]}, ultra thick, orange!90!black, dashed] (-1,1,1) to[bend right=20] 
            node[fill=white, below] {\small Step 3 ($X_2$)} (1,1,1);
        % Step 4: Flips Y (Moving to End)
        \draw[-{Latex[length=4mm]}, ultra thick, orange!90!black, dashed] (1,1,1) -- (1,-1,1) 
            node[fill=white, midway, right] {\small Step 4 ($Y_1$)};

    \end{tikzpicture}
    \caption{The tope graph of the $A_2$ arrangement embedded as a hexagon (blue/bold) inside the cube $Q_3$. Let $b^{\lambda}=(+\xrightarrow{X_1}-\xrightarrow{X_2}+)\otimes (+\xrightarrow{Y_1}-)\otimes (-\xrightarrow{Z_1} +)$, where $\lambda=(2,1,1)$. The 4-step shuffle walk $Z_1X_1X_2Y_1$ (orange/dotted) in $\nabla(b^\lambda)$ passes through the forbidden vertex $G^+=(+,+,+)$. Note also that this walk cannot appear in any other $\nabla(b^{\mu})$.}
    \label{fig:ghost_trap}
\end{figure}

\begin{rem}
    Corollary \ref{intMHbool} and Theorem \ref{nonbooldiagintMH} show that an arrangement $\A$ is Boolean if and only if $\beta_{\ell,\ell}^{\circ}(\A)\neq 0$ for some $\ell\geq 0$.
\end{rem}

\begin{cor}\label{diagbetti}
    Let $\A$ be an arrangement. The diagonal magnitude homology $MH_{\ell,\ell}(\A)$ is free abelian of rank $\beta_{\ell,\ell}(\A)$, which is determined by the Boolean flats of the intersection lattice $L(\A)$. More precisely, $\beta_{0,0}=\#\Ch(\A)$, and
    \[
    \beta_{\ell,\ell}(\A) = \sum_{\substack{X \in L(\A)\setminus\{\hat{0}\} \\ \rank X = \#\A_X}} c^X \cdot 2^{\rank X} \binom{\ell-1}{\rank X - 1}\quad (\ell\geq 1).
    \]
    Consequently, $\beta_{\ell,\ell}(\A)$ grows as a polynomial in $\ell$ of degree $d_{\max} - 1$, where
    \[
    d_{\max}=\max\{\rank X\mid X\in L(\A),\ \rank X=\#\A_X\}.
    \]
\end{cor}

\begin{proof}
    The claim for $\beta_{0,0}$ is proved in Proposition \ref{smalllength}.
    Let $\ell\geq 1$.
    Since $MC_{\ell+1,\ell}(\A)=0$, the diagonal magnitude homology $MH_{\ell,\ell}(\A)$ is free abelian. By Theorem \ref{facedecompMH}, the magnitude homology admits a face decomposition. Restricting this decomposition to the diagonal ($k=\ell$) and taking Betti numbers, we obtain
    \[
    \beta_{\ell,\ell}(\A) = \sum_{F \in \FF(\A)} \beta^\circ_{\ell,\ell}(\A_F).
    \]
    By Theorem \ref{nonbooldiagintMH}, the interior Betti number $\beta^\circ_{\ell,\ell}(\A_F)$ vanishes unless the localization $\A_F$ is Boolean, that is, $\#\A_F=\rank \A_F\geq 1$.
    Therefore, the sum collapses to the faces whose support flat $X = s(F)$ satisfies $\rank X = \#\A_X\geq 1$.
    For such faces, $\A_F \cong Bl_{\rank X}$.
    We can group the nonzero terms in the summation by their support flats $X \in L(\A)$.
    For each such flat $X$, the number of faces of $\A$ with support $X$ is exactly $c^X$.
    Substituting the known formula for the interior diagonal Betti numbers of the Boolean arrangement (Corollary \ref{intMHbool}), namely $\beta^\circ_{\ell,\ell}(Bl_d) = 2^d \binom{\ell-1}{d-1}$ where $d = \rank X$, we obtain
    \[
    \beta_{\ell,\ell}(\A) = \sum_{\substack{X \in L(\A)\setminus\{\hat{0}\} \\ \rank X = \#\A_X}} c^X \cdot 2^{\rank X} \binom{\ell-1}{\operatorname{rank}X - 1}.
    \]
    For $\ell\geq\rank X$, the binomial coefficient $\binom{\ell-1}{\rank X-1}$ is a polynomial in $\ell$ of degree $\rank X-1$. Therefore, for $\ell$ sufficiently large, $\beta_{\ell,\ell}(\A)$ grows as a polynomial in $\ell$ of degree $d_{\max}-1$.
\end{proof}

\begin{rem}
    Example \ref{hypercube} and Proposition \ref{smalllength} (2)(3) are special cases of the formula in Corollary \ref{diagbetti}.
\end{rem}

\section{Open problems}
Based on observations from examples in the appendix, we propose the following conjectures.

\begin{conj}\label{conjectures}
    Let $\A$ be an arrangement of rank $r$ and $n=\#\A$. Let $\Mag(\A,q)=P(q)/Q(q)$ be the reduced form. Then 
    \begin{enumerate}
        \item Suppose $r\geq 3$ is odd. Then $\T(\A)$ is not vertex-transitive if and only if $\Phi_{2n}(q)$ divides $Q(q)$.
        \item Suppose $r\geq 4$ is even. Then $\T(\A)$ is not vertex-transitive if and only if $\Phi_n(q)$ divides $Q(q)$.
        \item $(-1)^{\ell}\chi_\ell(\A)\geq0$ for sufficiently large $\ell$.
        \item $MH_{k,\ell}(\A)$ is determined by $L(\A)$.
        \item $MH_{k,\ell}(\A)$ is torsion-free.
        \item $MH^\circ_{k,\ell}(\A)\cong MH_{k-r,\ell-n}(\A)$ if $k\geq r$ and $\ell\geq n$, and $MH^\circ_{k,\ell}(\A)=0$ otherwise.
    \end{enumerate}
\end{conj}

\begin{rem}
    After this paper was made available as a preprint, the first author \cite{Koizumi2026} resolved Conjecture \ref{conjectures} (4)-(6) in the affirmative.
\end{rem}

\section*{Acknowledgements}
We thank Weili Guo, Yoh Kitajima, Takuya Saito, and Masahiko Yoshinaga for their comments on a preliminary version of this paper. We would also like to thank Yasuhiko Asao and Issei Enomoto for suggesting Remark \ref{steinberg}.

\appendix

\section{Computer experiments}\label{appendix}
The following examples were computed with SageMath \cite{sagemath}.
\begin{ex}\label{exa1}
We list the magnitudes of simplicial arrangements $\A(n,k)$ in Gr\"unbaum's catalogue \cite{Grunbaum2009} for $n\leq 13$. Note that $\A(n,0),\A(6,1),\A(9,1)$ are recorded in Examples \ref{nearpencil} and \ref{Coxeterrank3}. We use the database of \cite{Cuntz2022} to obtain the list of normal vectors.
\begin{align*}
    \Mag(\A(7,1))&=\frac{32q^4-40q^2+32}{(q^4+3q^3+4q^2+3q+1)(q^7+1)}\\
    &=\frac{8(4q^4-5q^2+4)}{[2]_q^2[3]_q(q^7+1)}
    =\frac{8(4q^4-5q^2+4)}{\Phi_2(q)^3\Phi_3(q)\Phi_{14}(q)}\\
    &=32-96q+120q^2-72q^3+24q^4-48q^5+144q^6-272q^7+360q^8-336q^9+240q^{10}+\cdots
\end{align*}
\begin{align*}
    \Mag(\A(8,1))&=\frac{40q^6-8q^4-16q^3-8q^2+40}{(q^6+3q^5+5q^4+6q^3+5q^2+3q+1)(q^8+1)}\\
    &=\frac{8(5q^6-q^4-2q^3-q^2+5)}{[4]_q!(q^8+1)}
    =\frac{8(5q^6-q^4-2q^3-q^2+5)}{\Phi_2(q)^2\Phi_3(q)\Phi_4(q)\Phi_{16}(q)}\\
    &=40-120q+152q^2-112q^3+88q^4-136q^5+240q^6-344q^7+352q^8-248q^9+176q^{10}+\cdots
\end{align*}
\begin{align*}
    \Mag(\A(10,1))&=\frac{60q^8+60q^7-20q^6-40q^5-40q^3-20q^2+60q+60}{(q^8+4q^7+8q^6+11q^5+12q^4+11q^3+8q^2+4q+1)(q^{10}+1)}\\
    &=\frac{20(3q^8+3q^7-q^6-2q^5-2q^3-q^2+3q+3)}{[2]_q^2[3]_q[5]_q(q^{10}+1)}\\
    &=\frac{20(3q^8+3q^7-q^6-2q^5-2q^3-q^2+3q+3)}{\Phi_2(q)^2\Phi_3(q)\Phi_4(q)\Phi_5(q)\Phi_{20}(q)}\\
    &=60-180q+220q^2-140q^3+60q^4-80q^5+220q^6-380q^7+420q^8-340q^9+220q^{10}+\cdots
\end{align*}
        
        \begin{align*}
        \Mag(\A(10,2))&=\Mag(\A(10,3))=\frac{60q^6-12q^4-48q^3-12q^2+60}{(q^6+3q^5+5q^4+6q^3+5q^2+3q+1)(q^{10}+1)}\\
        &=\frac{12(5q^6-2q^4-4q^3-2q^2+5)}{[4]_q!(q^{10}+1)}
        =\frac{12(5q^6-2q^4-4q^3-2q^2+5)}{\Phi_2(q)^2\Phi_3(q)\Phi_4(q)^2\Phi_{20}(q)}\\
        &=60-180q+228q^2-192q^3+204q^4-300q^5+432q^6-564q^7+660q^8-672q^9+576q^{10}+\cdots
        \end{align*}
        
        \begin{align*}
        \Mag(\A(11,1))&=\frac{72q^6-16q^4-64q^3-16q^2+72}{(q^6+3q^5+5q^4+6q^3+5q^2+3q+1)(q^{11}+1)}\\
        &=\frac{8(9q^6-2q^4-8q^3-2q^2+9)}{[4]_q!(q^{11}+1)}
        =\frac{8(9q^6-2q^4-8q^3-2q^2+9)}{\Phi_2(q)^3\Phi_3(q)\Phi_4(q)\Phi_{22}(q)}\\
        &=72-216q+272q^2-232q^3+256q^4-376q^5+528q^6-680q^7+800q^8-824q^9+784q^{10}+\cdots
    \end{align*}

    \begin{align*}
        \Mag(\A(12,1))&=\frac{84q^6-84q^5-36q^4+96q^3-36q^2-84q+84}{(q^6+2q^5+2q^4+2q^3+2q^2+2q+1)(q^{12}+1)}\\
        &=\frac{12(7q^6-7q^5-3q^4+8q^3-3q^2-7q+7)}{[2]_q[6]_q(q^{12}+1)}
        =\frac{12(7q^6-7q^5-3q^4+8q^3-3q^2-7q+7)}{\Phi_2(q)^2\Phi_3(q)\Phi_6(q)\Phi_8(q)\Phi_{24}(q)}\\
        &=84-252q+300q^2-168q^3+36q^4-84q^5+336q^6-588q^7+636q^8-504q^9+372q^{10}+\cdots
    \end{align*}

    \begin{align*}
        \Mag(\A(12,2))&=\frac{84q^{10}+84q^9+64q^8-4q^7-44q^6-128q^5-44q^4-4q^3+64q^2+84q+84}{(q^{10}+4q^9+9q^8+15q^7+20q^6+22q^5+20q^4+15q^3+9q^2+4q+1)(q^{12}+1)}\\
        &=\frac{4(21q^{10}+21q^9+16q^8-q^7-11q^6-32q^5-11q^4-q^3+16q^2+21q+21)}{[5]_q!(q^{12}+1)}\\
        &=\frac{4(21q^{10}+21q^9+16q^8-q^7-11q^6-32q^5-11q^4-q^3+16q^2+21q+21)}{\Phi_2(q)^2\Phi_3(q)\Phi_4(q)\Phi_5(q)\Phi_8(q)\Phi_{24}(q)}\\
        &=84-252q+316q^2-260q^3+252q^4-344q^5+508q^6-692q^7+804q^8-796q^9+760q^{10}+\cdots
    \end{align*}

    \begin{align*}
        \Mag(\A(12,3))&=\frac{84q^6-12q^4-96q^3-12q^2+84}{(q^6+3q^5+5q^4+6q^3+5q^2+3q+1)(q^{12}+1)}\\
        &=\frac{12(7q^6-q^4-8q^3-q^2+7)}{[4]_q!(q^{12}+1)}
        =\frac{12(7q^6-q^4-8q^3-q^2+7)}{\Phi_2(q)^2\Phi_3(q)\Phi_4(q)\Phi_8(q)\Phi_{24}(q)}\\
        &=84-252q+324q^2-312q^3+396q^4-564q^5+720q^6-876q^7+1044q^8-1128q^9+1116q^{10}+\cdots
    \end{align*}

    \begin{align*}
        \Mag(\A(13,1))&=\frac{96q^8-96q^7+72q^6-48q^5-48q^3+72q^2-96q+96}{(q^{8}+2q^{7}+3q^{6}+4q^{5}+4q^{4}+4q^{3}+3q^{2}+2q^{1}+1)(q^{13}+1)}\\
        &=\frac{24(4q^8-4q^7+3q^6-2q^5-2q^3+3q^2-4q+4)}{[2]_q^2(1+2q^2+2q^4+q^6)(q^{13}+1)}\\
        &=\frac{24(4q^8-4q^7+3q^6-2q^5-2q^3+3q^2-4q+4)}{\Phi_2(q)^3\Phi_3(q)\Phi_4(q)\Phi_6(q)\Phi_{26}(q)}\\
        &=96-288q+360q^2-288q^3+264q^4-384q^5+624q^6-864q^7+984q^8-960q^9+888q^{10}+\cdots
    \end{align*}

    \begin{align*}
        \Mag(\A(13,2))&=\frac{96q^6-144q^3+96}{(q^6+3q^5+5q^4+6q^3+5q^2+3q+1)(q^{13}+1)}\\
        &=\frac{48(2q^6-3q^3+2)}{[4]_q!(q^{13}+1)}
        =\frac{48(2q^6-3q^3+2)}{\Phi_2(q)^3\Phi_3(q)\Phi_4(q)\Phi_{26}(q)}\\
        &=96-288q+384q^2-432q^3+624q^4-864q^5+1008q^6-1152q^7+1392q^8-1584q^9+1632q^{10}+\cdots
    \end{align*}

    \begin{align*}
        \Mag(\A(13,3))&=\frac{96q^{10}+96q^9+80q^8-8q^7-64q^6-160q^5-64q^4-8q^3+80q^2+96q+96}{(q^{10}+4q^9+9q^8+15q^7+20q^6+22q^5+20q^4+15q^3+9q^2+4q+1)(q^{13}+1)}\\
        &=\frac{8(12q^{10}+12q^9+10q^8-q^7-8q^6-20q^5-8q^4-q^3+10q^2+12q+12)}{[5]_q!(q^{13}+1)}\\
        &=\frac{8(12q^{10}+12q^9+10q^8-q^7-8q^6-20q^5-8q^4-q^3+10q^2+12q+12)}{\Phi_2(q)^3\Phi_3(q)\Phi_4(q)\Phi_5(q)\Phi_{26}(q)}\\
        &=96-288q+368q^2-328q^3+336q^4-424q^5+584q^6-784q^7+912q^8-920q^9+920q^{10}+\cdots
    \end{align*}

    \begin{align*}
        \Mag(\A(13,4))&=\frac{104q^6-56q^4-48q^3-56q^2+104}{(q^6+3q^5+5q^4+6q^3+5q^2+3q+1)(q^{13}+1)}\\
        &=\frac{8(13q^6-7q^4-6q^3-7q^2+13)}{[4]_q!(q^{13}+1)}
        =\frac{8(13q^6-7q^4-6q^3-7q^2+13)}{\Phi_2(q)^3\Phi_3(q)\Phi_4(q)\Phi_{26}(q)}\\
        &=104-312q+360q^2-192q^3+72q^4-168q^5+432q^6-696q^7+792q^8-672q^9+504q^{10}+\cdots
    \end{align*}
\end{ex}

\begin{ex}\label{generic4}
    Consider the rank $3$ generic arrangement $\A=U_{3,4}$ defined by $xyz(x+y+z)$.
    \begin{align*}
        \Mag(\A)&=\frac{14q^2-20q+14}{(q^2+2q+1)(q^4+1)}\\
        &=\frac{2(7q^2-10q+7)}{[2]_q^2(q^4+1)}
        =\frac{2(7q^2-10q+7)}{\Phi_2(q)^2\Phi_8(q)}\\
        &=14-48q+96q^2-144q^3+178q^4-192q^5+192q^6-192q^7+206q^8-240q^9+288q^{10}+\cdots
    \end{align*}
\end{ex}

\begin{ex}\label{K4-e}
    Consider the non-simplicial arrangement $\A$ defined by $xyz(x+z)(y+z)$, which is also the graphic arrangement of the graph $K_4\setminus e$. Note that $K_4\setminus e$ is chordal.
    \begin{align*}
        \Mag(\A)&=\frac{18q^4-2q^3-8q^2-2q+18}{(q^4+3q^3+4q^2+3q+1)(q^5+1)}\\
        &=\frac{2(9q^4-q^3-4q^2-q+9)}{[2]_q^2[3]_q(q^5+1)}
        =\frac{2(9q^4-q^3-4q^2-q+9)}{\Phi_2(q)^3\Phi_3(q)\Phi_{10}(q)}\\
        &=18-56q+88q^2-96q^3+104q^4-154q^5+248q^6-336q^7+376q^8-392q^9+450q^{10}+\cdots
    \end{align*}
\end{ex}

\begin{ex}
    Consider the graphic arrangement $\A$ of the graph $K_5\setminus e$.
    \begin{align*}
        \Mag(\A)&=\frac{12(8q^8-9q^7+4q^6+3q^5+3q^4+4q^2-9q+8)}{[2]_q[4]_q[9]_q(q^5+1)}\\
        &=\frac{12(8q^8-9q^7+4q^6+3q^5+3q^4+4q^2-9q+8)}{\Phi_2(q)^3\Phi_3(q)\Phi_4(q)\Phi_9(q)\Phi_{10}(q)}\\
        &=96-396q+756q^2-924q^3+996q^4-1320q^5+1956q^6-2652q^7+3252q^8-3756q^9+4212q^{10}+\cdots
    \end{align*}
    where the numerator is $P(q)=96q^8-108q^7+48q^6+36q^5+36q^3+48q^2-108q+96$ and the denominator is $Q(q)=q^{17}+3q^{16}+5q^{15}+7q^{14}+8q^{13}+9q^{12}+11q^{11}+13q^{10}+15q^9+15q^8+13q^7+11q^6+9q^5+8q^4+7q^3+5q^2+3q+1$.
\end{ex}

\begin{ex}
    Consider the bracelet arrangement $\A$ defined by $x_1x_2x_3(x_1+x_4)(x_2+x_4)(x_3+x_4)(x_1+x_2+x_4)(x_1+x_3+x_4)(x_2+x_3+x_4)$ (see Example 1.4 of \cite{Denham2026}).
    \begin{align*}
        \Mag(\A)&=\frac{6(17q^8-4q^7-8q^6-9q^5+44q^4-9q^3-8q^2-4q+17)}{\Phi_2(q)^3\Phi_3(q)^2\Phi_9(q)\Phi_{10}(q)}\\
        &=102-432q+864q^2-1176q^3+1584q^4-2628q^5+4344q^6-6048q^7+7200q^8-8130q^9+9684q^{10}+\cdots
    \end{align*}
    where the numerator is $P(q)=102q^8-24q^7-48q^6-54q^5+264q^4-54q^3-48q^2-24q+102$ and the denominator is $Q(q)=q^{17}+4q^{16}+8q^{15}+11q^{14}+12q^{13}+13q^{12}+16q^{11}+20q^{10}+23q^9+23q^8+20q^7+16q^6+13q^5+12q^4+11q^3+8q^2+4q+1$.
\end{ex}

\begin{ex}\label{ex:MH_tables}
Here we collect magnitude Betti numbers computed using the SageMath code of \cite{Hepworth2017}.
\end{ex}

\begin{table}[htpb]
\centering
\begin{tabular}{c|ccccccccc}
\hline
$k \backslash \ell$ & 0 & 1 & 2 & 3 & 4 & 5 & 6 & 7 & 8 \\
\hline
0 & 4 & 0  & 0  & 0   & 0   & 0   & 0   & 0  & 0 \\
1 & 0  & 8 & 0  & 0   & 0   & 0   & 0   & 0  & 0\\
2 & 0  & 0  & 12 & 0  & 0   & 0   & 0   & 0  & 0\\
3 & 0  & 0  & 0  & 16 & 0  & 0  & 0   & 0  & 0\\
4 & 0  & 0  & 0  & 0   & 20 & 0  & 0  & 0  & 0\\
5 & 0  & 0  & 0  & 0   & 0   & 24 & 0  & 0  & 0\\
6 & 0  & 0  & 0  & 0   & 0   & 0   & 28 & 0  & 0\\
7 & 0  & 0  & 0  & 0   & 0   & 0   & 0   & 32 & 0\\
8 & 0  & 0  & 0  & 0   & 0   & 0   & 0   & 0 & 36\\
\hline
\end{tabular}
\vspace{1ex}
\caption{Magnitude Betti numbers $\beta_{k,\ell}$ for the Boolean arrangement $Bl_2=U_{2,2}$ defined by $xy=0$.}
\label{tab:Bool2}
\end{table}

\begin{table}[htpb]
\centering
\begin{tabular}{c|ccccccccc}
\hline
$k \backslash \ell$ & 0 & 1 & 2 & 3 & 4 & 5 & 6 & 7 & 8 \\
\hline
0 & 6 & 0  & 0  & 0   & 0   & 0   & 0   & 0  & 0 \\
1 & 0  & 12 & 0  & 0   & 0   & 0   & 0   & 0  & 0\\
2 & 0  & 0  & 12 & 6  & 0   & 0   & 0   & 0  & 0\\
3 & 0  & 0  & 0  & 12 & 12  & 0  & 0   & 0  & 0\\
4 & 0  & 0  & 0  & 0   & 12 & 12  & 6  & 0  & 0\\
5 & 0  & 0  & 0  & 0   & 0   & 12 & 12  & 12  & 0\\
6 & 0  & 0  & 0  & 0   & 0   & 0   & 12 & 12  & 12\\
7 & 0  & 0  & 0  & 0   & 0   & 0   & 0   & 12 & 12\\
8 & 0  & 0  & 0  & 0   & 0   & 0   & 0   & 0 & 12\\
\hline
\end{tabular}
\vspace{1ex}
\caption{Magnitude Betti numbers $\beta_{k,\ell}$ for the Coxeter arrangement $A_2=U_{2,3}$.}
\label{tab:A2}
\end{table}

\begin{table}[htpb]
\centering
\begin{tabular}{c|ccccccccc}
\hline
$k \backslash \ell$ & 0 & 1 & 2 & 3 & 4 & 5 & 6 & 7 & 8 \\
\hline
0 & 8 & 0  & 0  & 0   & 0   & 0   & 0   & 0  & 0 \\
1 & 0  & 16 & 0  & 0   & 0   & 0   & 0   & 0  & 0\\
2 & 0  & 0  & 16 & 0  & 8   & 0   & 0   & 0  & 0\\
3 & 0  & 0  & 0  & 16 & 0  & 16  & 0   & 0  & 0\\
4 & 0  & 0  & 0  & 0   & 16 & 0  & 16  & 0  & 8\\
5 & 0  & 0  & 0  & 0   & 0   & 16 & 0  & 16  & 0\\
6 & 0  & 0  & 0  & 0   & 0   & 0   & 16 & 0  & 16\\
7 & 0  & 0  & 0  & 0   & 0   & 0   & 0   & 16 & 0\\
8 & 0  & 0  & 0  & 0   & 0   & 0   & 0   & 0 & 16\\
\hline
\end{tabular}
\vspace{1ex}
\caption{Magnitude Betti numbers $\beta_{k,\ell}$ for the Coxeter arrangement $B_2=U_{2,4}$.}
\label{tab:B2}
\end{table}

\begin{table}[htpb]
\centering
\begin{tabular}{c|ccccccccc}
\hline
$k \backslash \ell$ & 0 & 1 & 2 & 3 & 4 & 5 & 6 & 7 & 8 \\
\hline
0 & 24 & 0  & 0  & 0   & 0   & 0   & 0   & 0  & 0 \\
1 & 0  & 72 & 0  & 0   & 0   & 0   & 0   & 0  & 0\\
2 & 0  & 0  & 96 & 48  & 0   & 0   & 0   & 0  & 0\\
3 & 0  & 0  & 0  & 120 & 96  & 0  & 24   & 0  & 0\\
4 & 0  & 0  & 0  & 0   & 144 & 96  & 48  & 72  & 0\\
5 & 0  & 0  & 0  & 0   & 0   & 168 & 96  & 96  & 96\\
6 & 0  & 0  & 0  & 0   & 0   & 0   & 192 & 96  & 96\\
7 & 0  & 0  & 0  & 0   & 0   & 0   & 0   & 216 & 96\\
8 & 0  & 0  & 0  & 0   & 0   & 0   & 0   & 0 & 240\\
\hline
\end{tabular}
\vspace{1ex}
\caption{Magnitude Betti numbers $\beta_{k,\ell}$ for the Coxeter arrangement $A_3$.}
\label{tab:A3}
\end{table}

\begin{table}[htpb]
\centering
\begin{tabular}{c|ccccccc}
\hline
$k \backslash \ell$ & 0 & 1 & 2 & 3 & 4 & 5 & 6  \\
\hline
0 & 120 & 0  & 0  & 0   & 0   & 0    & 0  \\
1 & 0  & 480 & 0  & 0   & 0   & 0     & 0 \\
2 & 0  & 0  & 840 & 360  & 0   & 0    & 0  \\
3 & 0  & 0  & 0  & 1200 & 960  & 0   & 240  \\
4 & 0  & 0  & 0  & 0   & 1560 & 1440  & 360  \\
5 & 0  & 0  & 0  & 0   & 0   & 1920   & 1920 \\
6 & 0  & 0  & 0  & 0  & 0   &  0   & 2280\\
\hline
\end{tabular}
\vspace{1ex}
\caption{Magnitude Betti numbers $\beta_{k,\ell}$ for the Coxeter arrangement $A_4$.}
\label{tab:A4}
\end{table}

\begin{table}[htpb]
\centering
\begin{tabular}{c|ccccccccc}
\hline
$k \backslash \ell$ & 0 & 1 & 2 & 3 & 4 & 5 & 6 & 7 & 8 \\
\hline
0 & 14 & 0  & 0  & 0   & 0   & 0   & 0   & 0   &  0\\
1 & 0  & 48 & 0  & 0   & 0   & 0   & 0   & 0   &  0\\
2 & 0  & 0  & 96 & 0  & 0   & 0   & 0   & 0   &  0\\
3 & 0  & 0  & 0  & 144 & 14  & 0  & 0   & 0   & 0\\
4 & 0  & 0  & 0  & 0   & 192 & 48  & 0  & 0   & 0\\
5 & 0  & 0  & 0  & 0   & 0   & 240 & 96  & 0 & 0\\
6 & 0  & 0  & 0  & 0   & 0   & 0   & 288 & 144  & 14\\
7 & 0  & 0  & 0  & 0   & 0   & 0   & 0   & 336 & 192\\
8 & 0  & 0  & 0  & 0   & 0   & 0   & 0   & 0   & 384\\ 
\hline
\end{tabular}
\vspace{1ex}
\caption{Magnitude Betti numbers $\beta_{k,\ell}$ for the generic arrangement $U_{3,4}$ defined by $xyz(x+y+z)$; see Example \ref{generic4}.}
\label{tab:generic4}
\end{table}

\begin{table}[htpb]
\centering
\begin{tabular}{c|ccccccc}
\hline
$k \backslash \ell$ & 0 & 1 & 2 & 3 & 4 & 5 & 6  \\
\hline
0 & 30 & 0  & 0  & 0   & 0   & 0    & 0  \\
1 & 0  & 140 & 0  & 0   & 0   & 0     & 0 \\
2 & 0  & 0  & 380 & 0  & 0   & 0    & 0  \\
3 & 0  & 0  & 0  & 780 & 0  & 0   & 0  \\
4 & 0  & 0  & 0  & 0   & 1340 & 30  & 0  \\
5 & 0  & 0  & 0  & 0   & 0   & 2060   & 140 \\
6 & 0  & 0  & 0  & 0  & 0   &  0   & 2940\\
\hline
\end{tabular}
\vspace{1ex}
\caption{Magnitude Betti numbers $\beta_{k,\ell}$ for the generic arrangement $U_{4,5}$.}
\label{tab:U45}
\end{table}

\begin{table}[htpb]
\centering
\begin{tabular}{c|ccccccccc}
\hline
$k \backslash \ell$ & 0 & 1 & 2 & 3 & 4 & 5 & 6 & 7 & 8 \\
\hline
0 & 18 & 0  & 0  & 0   & 0   & 0   & 0   & 0   &  0\\
1 & 0  & 56 & 0  & 0   & 0   & 0   & 0   & 0   &  0\\
2 & 0  & 0  & 88 & 24  & 0   & 0   & 0   & 0   &  0\\
3 & 0  & 0  & 0  & 120 & 48  & 18  & 0   & 0   & 0\\
4 & 0  & 0  & 0  & 0   & 152 & 48  & 80  & 0   & 0\\
5 & 0  & 0  & 0  & 0   & 0   & 184 & 48  & 136 & 24\\
6 & 0  & 0  & 0  & 0   & 0   & 0   & 216 & 48  & 168\\
7 & 0  & 0  & 0  & 0   & 0   & 0   & 0   & 248 & 48\\
8 & 0  & 0  & 0  & 0   & 0   & 0   & 0   & 0   & 280\\ 
\hline
\end{tabular}
\vspace{1ex}
\caption{Magnitude Betti numbers $\beta_{k,\ell}$ for the graphic arrangement of $K_4 \setminus e$; see Example \ref{K4-e}.}
\label{tab:k4_minus_e}
\end{table}

\begin{table}[htpb]
\centering
\begin{tabular}{c|ccccccccc}
\hline
$k \backslash \ell$ & 0 & 1 & 2 & 3 & 4 & 5 & 6 & 7 & 8 \\
\hline
0 & 32 & 0  & 0  & 0   & 0   & 0   & 0   & 0   &  0\\
1 & 0  & 96 & 0  & 0   & 0   & 0   & 0   & 0   &  0\\
2 & 0  & 0  & 120 & 72  & 0   & 0   & 0   & 0   &  0\\
3 & 0  & 0  & 0  & 144 & 144  & 0  & 0   & 32   & 0\\
4 & 0  & 0  & 0  & 0   & 168 & 144  & 72  & 0   & 96\\
5 & 0  & 0  & 0  & 0   & 0   & 192 & 144  & 144 & 0\\
6 & 0  & 0  & 0  & 0   & 0   & 0   & 216 & 144  & 144\\
7 & 0  & 0  & 0  & 0   & 0   & 0   & 0   & 240 & 144\\
8 & 0  & 0  & 0  & 0   & 0   & 0   & 0   & 0   & 264\\ 
\hline
\end{tabular}
\vspace{1ex}
\caption{Magnitude Betti numbers $\beta_{k,\ell}$ for $\A(7,1)$}
\label{tab:A71}
\end{table}

\begin{table}[htpb]
\centering
\begin{tabular}{c|ccccccccc}
\hline
$k \backslash \ell$ & 0 & 1 & 2 & 3 & 4 & 5 & 6 & 7 & 8 \\
\hline
0 & 40 & 0  & 0  & 0   & 0   & 0   & 0   & 0   &  0\\
1 & 0  & 120 & 0  & 0   & 0   & 0   & 0   & 0   &  0\\
2 & 0  & 0  & 152 & 72  & 16   & 0   & 0   & 0   &  0\\
3 & 0  & 0  & 0  & 184 & 144  & 32  & 0   & 0   & 40\\
4 & 0  & 0  & 0  & 0   & 216 & 144  & 104  & 0   & 16\\
5 & 0  & 0  & 0  & 0   & 0   & 248 & 144  & 176 & 0\\
6 & 0  & 0  & 0  & 0   & 0   & 0   & 280 & 144  & 176\\
7 & 0  & 0  & 0  & 0   & 0   & 0   & 0   & 312 & 144\\
8 & 0  & 0  & 0  & 0   & 0   & 0   & 0   & 0   & 344\\ 
\hline
\end{tabular}
\vspace{1ex}
\caption{Magnitude Betti numbers $\beta_{k,\ell}$ for $\A(8,1)$}
\label{tab:A81}
\end{table}

\begin{table}[htpb]
\centering
\begin{tabular}{c|ccccccccc}
\hline
$k \backslash \ell$ & 0 & 1 & 2 & 3 & 4 & 5 & 6 & 7 & 8 \\
\hline
0 & 48 & 0  & 0  & 0   & 0   & 0   & 0   & 0   &  0\\
1 & 0  & 144 & 0  & 0   & 0   & 0   & 0   & 0   &  0\\
2 & 0  & 0  & 192 & 48  & 48   & 0   & 0   & 0   &  0\\
3 & 0  & 0  & 0  & 240 & 96  & 96  & 0   & 0   & 0\\
4 & 0  & 0  & 0  & 0   & 288 & 96  & 144  & 0   & 48\\
5 & 0  & 0  & 0  & 0   & 0   & 336 & 96  & 192 & 0\\
6 & 0  & 0  & 0  & 0   & 0   & 0   & 384 & 96  & 192\\
7 & 0  & 0  & 0  & 0   & 0   & 0   & 0   & 432 & 96\\
8 & 0  & 0  & 0  & 0   & 0   & 0   & 0   & 0   & 480\\ 
\hline
\end{tabular}
\vspace{1ex}
\caption{Magnitude Betti numbers $\beta_{k,\ell}$ for $\A(9,1)$}
\label{tab:A91}
\end{table}

\FloatBarrier

\begin{rem}
In almost all of the examples appearing in Example \ref{ex:MH_tables}, the equality
\[
\beta_{0,0}=\beta_{r,n}
\]
holds, where $r=\rank \A$ and $n=\#\A$. It is therefore natural to ask how generally this equality holds. Of course, it does not hold for Boolean arrangements. Moreover, further examples can be constructed by taking the direct sum of two or more arrangements. For instance, for $\A=U_{2,3}\oplus Bl_1$, the K\"unneth formula gives
\[
\beta_{3,4}(\A)
=
\beta_{3,4}(U_{2,3})\beta_{0,0}(Bl_1)
+
\beta_{2,3}(U_{2,3})\beta_{1,1}(Bl_1)
=36,
\]
whereas $\beta_{0,0}(\A)=12$.

There even exist examples with no nontrivial direct-sum decomposition, as the following example shows.
Let $G_{\house}$ be the house graph, that is, $V(G_\house)=\{1,2,3,4,5\}$ and $E(G_\house)=\{12,23,34,45,51,25\}$ (see Figure \ref{fig:Ghouse}).
Let $\A(G_\house)$ be the graphical arrangement associated with $G_\house$, i.e., $\A(G_\house)=\{x_i=x_j\mid ij\in E(G_\house)\}$.
Then we have $\rank \A(G_\house)=4$ and $\#\A(G_\house)=6$, but $\beta_{0,0}=42\neq 78=\beta_{4,6}$ (see Table \ref{tab:large_beta_r_n}).
\end{rem}

\begin{figure}[htpb]
\centering
\begin{tikzpicture}[scale=1.2,
  vertex/.style={circle, draw, fill=white, inner sep=1.5pt}
]
  \node[vertex] (1) at (1,2.5) {$1$};
  \node[vertex] (2) at (0,1.5) {$2$};
  \node[vertex] (3) at (0,0) {$3$};
  \node[vertex] (4) at (2,0) {$4$};
  \node[vertex] (5) at (2,1.5) {$5$};
  
  \draw (1)--(2)--(3)--(4)--(5)--(1);
  \draw (2)--(5);
\end{tikzpicture}
\vspace{1ex}
\caption{The house graph $G_\house$}
\label{fig:Ghouse}
\end{figure}

\begin{table}[htpb]
\centering
\begin{tabular}{c|ccccccc}
\hline
$k \backslash \ell$ & 0 & 1 & 2 & 3 & 4 & 5 & 6\\
\hline
0 & 42 & 0  & 0  & 0   & 0   & 0   & 0\\
1 & 0  & 188& 0  & 0   & 0   & 0   & 0\\
2 & 0  & 0  & 452& 36  & 0   & 0   & 0\\
3 & 0  & 0  & 0  & 812 & 172 & 0   & 0\\
4 & 0  & 0  & 0  & 0   & 1268& 384 & 78\\
5 & 0  & 0  & 0  & 0   & 0   & 1820& 624\\
6 & 0  & 0  & 0  & 0   & 0   & 0   & 2468\\
\hline
\end{tabular}
\vspace{1ex}
\caption{Magnitude Betti numbers $\beta_{k,\ell}$ for $\A(G_\house)$}
\label{tab:large_beta_r_n}
\end{table}

\FloatBarrier

\renewcommand\refname{References}
\bibliographystyle{hep}
\bibliography{magarr}

\end{document}